\documentclass[10pt,english,leqno]{siamltex1213}
\usepackage[T1]{fontenc}
\usepackage[latin9]{inputenc}
\usepackage{amsmath}
\usepackage{amssymb}

\makeatletter
\usepackage[left=3.0cm,right=3.0cm,top=2.2cm,bottom=2.2cm]{geometry}

\title{The existence and uniqueness of a power price equilibrium} 

\author{Miha Troha\thanks{Mathematical Institute, Oxford University, Andrew Wiles Building, Radcliffe Observatory Quarter, Woodstock Road, Oxford OX2 6GG, United Kingdom, \email{troha@maths.ox.ac.uk}. This author was supported through grants from the Slovene human resources development and scholarship fund, and the Oxford-Man Institute.} \and Raphael Hauser\thanks{Mathematical Institute, Oxford University, Andrew Wiles Building, Radcliffe Observatory Quarter, Woodstock Road, Oxford OX2 6GG, United Kingdom, \email{hauser@maths.ox.ac.uk}. Associate Professor in Numerical Mathematics, and Tanaka Fellow in Applied Mathematics at Pembroke College, Oxford. This author was supported through grant EP/H02686X/1 from the Engineering and Physical Sciences Research Council of the UK.}}

\usepackage{tikz}
\usepackage{verbatim}
\usepackage{varwidth}
\usepackage{ragged2e}
\usetikzlibrary{shapes,arrows,positioning,fit,calc,backgrounds,decorations.shapes}

\tikzset{decorate sep/.style 2 args=
{decorate,decoration={shape backgrounds,shape=circle,shape size=#1,shape sep=#2}}}

\makeatother

\usepackage{babel}
\begin{document}
\maketitle
\begin{abstract}
We propose a term structure power price model that, in contrast to
widely accepted no-arbitrage based approaches, accounts for the non-storable
nature of power. It belongs to a class of equilibrium game theoretic
models with players divided into producers and consumers. The consumers'
goal is to maximize a mean-variance utility function subject to satisfying
an inelastic demand of their own clients (e.g households, businesses
etc.) to whom they sell the power. The producers, who own a portfolio
of power plants each defined by a running fuel (e.g. gas, coal, oil...)
and physical characteristics (e.g. efficiency, capacity, ramp up/down
times...), similarly, seek to maximize a mean-variance utility function
consisting of power, fuel, and emission prices subject to production
constraints. Our goal is to determine the term structure of the power
price at which production matches consumption. In this paper we show
that in such a setting the equilibrium price exists and also discuss
the conditions for its uniqueness.
\end{abstract}
\begin{keywords}{\footnotesize term structure, electricity, game
theory, mean-variance, optimization, KKT conditions.}\end{keywords}

\section{Introduction}

The electricity price has some unique features which are very distinct
from features of stock prices or other commodities. The literature
aiming at modeling the electricity price can be broadly divided into
non-structural, structural, and game theoretic approaches.

Non-structural approaches model electricity price directly, without
considering the underlying reasons that cause the price to change
over time. A good overview of regular patterns and statistical properties
of electricity prices, together with one and multi-factor Ornstein-Uhlenbeck
process models of the spot price, is given in \cite{lucia2000electricity}.
Since Gaussian distributions are not suitable for modeling spikes,
a combination of Ornstein-Uhlenbeck and pure jump-processes have been
proposed in \cite{hambly2009modelling}. Further developments include
using also Levy processes (see \cite{meyer-brandis2008multifactor}
for example). Even though such models are good for modeling the spot
price, their application for pricing derivatives is sometimes rather
questionable. As pointed out in \cite{benth2009theinformation}, a
classical buy-and-hold strategy is not applicable to non-storable
commodities. All non-structural models define the price of a forward
contract as the expected price of the security at the delivery under
the risk-neutral measure, conditioned on some filtration, which contains
all available market information. It is often assumed that all available
market information is included in the spot price (i.e. the spot price
is a martingale). This is a good approximation for stock prices, but
definitely false for non-storable commodities. Information about a
huge electricity demand increase in a week time, for example, does
not have any impact on today's spot price. The correct filtration
must hence include information on weather forecasts, planned power
plant outages etc. It can be argued that all the information is contained
in prices of forward contracts and other derivatives. A one-factor
model in \cite{clewlow1999valuing} and multi-factor term-structure
model in \cite{clewlow1999amultifactor} are the first that produce
prices that are consistent with observable forward prices.

Another interesting aspect of electricity is its relation to other
fuels. Noting that coal, gas, oil, etc. can be used to produce electricity,
electricity itself can be treated as a derivative with various fuel
and emission prices as the underlying. \cite{barlow2002adiffusion}
produced seminal work that used a supply and demand stack to model
electricity prices. The supply stack was further extended by \cite{howison2009stochastic}
and by \cite{carmona2013electricity} to include various fuels. Using
an exponential bid stack, they calculated spot electricity price as
a function of demand and spot fuel prices. These models belong to
a class of structural approaches since they try to capture some of
the physical properties of the electricity markets.

The third, game theoretic approach, models the physical properties
and decisions of market participants more closely. As pointed out
in \cite{robinson2005mathmodel}, models that include ramp up/down
constraints of power plants and study the impact of long term contracts
on the spot electricity prices, are needed in order to prevent and
explain disastrous events an example of which happened in California
in 2001 and cost the state as much as \$45 billion. The seminal work
of game theoretic models in electricity markets was produced by \cite{bessembinder2002equilibrium}:
a unique relation between a forward and a spot price is given in a
two-stage market with one producer and one consumer, who each want
to maximize their mean-variance objective function. This work has
been used to study benefits of derivatives in the electricity market,
see \cite{cavallo2005electricity}. It was extended to a multi-stage
setting with a dynamic equilibrium by \cite{buhler2009valuation},
and \cite{buhler2009riskpremia}. \cite{demaeredaertrycke2012liquidity}
have extended the two-stage mean-variance model to any convex risk
measure, while also taking into account liquidity constraints. 

The game theoretic approach attempts to model decisions of producers
and consumers that participate in the electricity market explicitly.
Since a widely used strategy for risk management is delta hedging,
it is important to know that the delta hedging strategy can be implemented
as a minimum variance strategy in the mean-variance portfolio framework
(see \cite{alexander2006hedging} for details).

Our work belongs to a class of game theoretic approaches. It extends
the model of \cite{buhler2009riskpremia} to more than one producer
and consumer, who maximize their mean-variance utility functions.
In contrast to other game theoretic models, we include capacity and
ramp up/down constraints of power plants. Following ideas from structural
approaches, the profit of power plants is modeled as a difference
between electricity price and fuel costs (including emissions). Since
we model each power plant directly, we do not have to make any assumptions
on the bid stack, which is in our case determined by the physical
properties of power plants. As in \cite{clewlow1999amultifactor},
our model is consistent with observable fuel and emission prices.
We do not focus on specific emission market schemes, but rather use
a simplified version of it. For a detailed treatment of emission market
in the game theoretic setting see \cite{carmona2010marketdesign}
and \cite{ludkovski2011stochastic}.

The literature distinguishes between dynamic (i.e. stochastic) and
static models. In dynamic models, players adapt to changing environment
(i.e. fuel prices, demand etc.) by adapting their decisions. In each
stage of their decision making they determine the optimal decisions
they have to take now and also the optimal decisions they will take
in the future under all possible changes in the environment. Since
future decisions affect present decisions this is computationally
very demanding. In static models on the other hand, players assume
that they know the future state of the environment exactly and can
thus stick to an initial plan about future decisions, regardless of
the changes in the environment. Such approaches are computationally
much more tractable, but they do not reflect the reality very well.
The model we propose here can be seen as hybrid of both approaches.
The initial optimization problem is static, where players determine
all their optimal decisions and assume they will not alter them in
the future. However, as the environment changes, players may take
recursive actions by calculating new optimal decisions while taking
into account the new state of the environment, as well as their decisions
taken under the previous state of the environment. Since in such setting
the recursive actions are not very dependent, we will mainly focus
on each of them separately.

This paper is organized as follows: In Section \ref{sec:Problem-description}
we give a detailed mathematical description of the model, and in Section
\ref{sec:Analysis} we proof that the solution to our model exists
and develop the conditions under which it is also unique. We conclude
the paper in Section \ref{sec:Conclusions}.

\section{Problem description\label{sec:Problem-description}}

The model we propose can be used to determine the term structure of
electricity prices. The electricity market is defined by a set of
producers $P$ of cardinality $0<\left|P\right|<\infty$, a set of
consumers $C$ of cardinality $0<\left|C\right|<\infty$, and a hypothetical
market agent. Each of the producers and consumers participates in
the electricity market in order to maximize their profit subject to
a risk budget under a mean-variance optimization framework. Producers
own a number of power plants, which can have different physical characteristics
and run on different fuels. The set of all fuels is denoted by $L$.
Sets $R^{p,l}$ represent all power plants owned by producer $p\in P$
that run on fuel $l\in L$.

We are interested in delivery times $T_{j}$, $j\in J=\left\{ 1,...,T'\right\} $,
where power for each delivery time $T_{j}$ can be traded through
numerous forward contracts at times $t_{i}$, $i\in I_{j}$. The electricity
price at time $t_{i}$ for delivery at time $T_{j}$ is denoted by
$\Pi\left(t_{i},T_{j}\right)$. Since contracts with trading time
later than delivery time do not exist, we require $t_{\max\left\{ I_{j}\right\} }=T_{j}$
for all $j\in J$. The number of all forward contracts, i.e. $\sum_{j\in J}\left|I_{j}\right|$,
is denoted by $N$. Uncertainty is modeled by a filtered probability
space $\left(\Omega,\mathcal{F},\mathbb{F}=\left\{ \mathcal{F}_{t},t\in I\right\} ,\mathbb{P}\right)$,
where $I=\cup_{j\in J}I_{j}$. The $\sigma$-algebra $\mathcal{F}_{t}$
represents information available at time $t$.

The exogenous variables that appear in our model are (a) aggregate
power demand $D\left(T_{j}\right)$ for each delivery period $j\in J$,
(b) prices of fuel forward contracts $G_{l}\left(t_{i},T_{j}\right)$
for each fuel $l\in L$, delivery period $j\in J$, and trading period
$i\in I_{j}$, and (c) prices of emissions forward contracts $G_{em}\left(t_{i},T_{j}\right)$,
$j\in J$, $i\in I_{j}$. Electricity prices and all exogenous variables
are assumed to be adapted to the filtration $\left\{ \mathcal{F}_{t}\right\} _{t\in I}$
and have finite second moments.

Let $v_{k}\in\mathbb{R}^{n_{k}}$, $n_{k}\in\mathbb{N}$, $k\in K$,
and $K=\left\{ 1,...,\left|K\right|\right\} $ be given vectors. For
convenience, we define a vector concatenation operator as
\[
\left|\right|_{k\in K}v_{k}=\left[v_{1}^{\top},...,v_{\left|K\right|}^{\top}\right]^{\top}.
\]

\subsection{Producer}

Each producer $p\in P$ has to decide on the number $V_{p}\left(t_{i},T_{j}\right)$
of electricity forward contracts and the number $F_{p,l}\left(t_{i},T_{j}\right)$,
$l\in L$ of fuel forward contracts to buy, at trading time $t_{i}$,
$i\in I_{j}$ for delivery at time $T_{j}$, $j\in J$. By $W_{p,l,r}\left(T_{j}\right)$
we denote the actual production from fuel $l\in L$ at time $T_{j}$
from power plant $r\in R^{p,l}$, and by $O_{p}\left(t_{i},T_{j}\right)$
we denote the number of emission forward contracts purchased at time
$t_{i}$ for delivery at time $T_{j}$.

Notation is greatly simplified if decision variables are concatenated
into 
\begin{itemize}
\item electricity trading vectors $V_{p}\left(T_{j}\right)=\left|\right|_{i\in I_{j}}V_{p}\left(t_{i},T_{j}\right)$
and $V_{p}=\left|\right|_{j\in J}V_{p}\left(T_{j}\right)$, 
\item fuel trading vectors $F_{p}\left(t_{i},T_{j}\right)=\left|\right|_{l\in L}F_{p,l}\left(t_{i},T_{j}\right)$,
$F_{p}\left(T_{j}\right)=\left|\right|_{i\in I_{j}}F_{p}\left(t_{i},T_{j}\right)$,
and\linebreak{}
 $F_{p}=\left|\right|_{j\in J}F_{p}\left(T_{j}\right)$, 
\item emission trading vectors $O_{p}\left(T_{j}\right)=\left|\right|_{i\in I_{j}}O_{p}\left(t_{i},T_{j}\right)$
and $O_{p}=\left|\right|_{j\in J}O_{p}\left(T_{j}\right)$, 
\item electricity production vectors $W_{p,l}\left(T_{j}\right)=\left|\right|_{r\in R^{p,l}}W_{p,l,r}\left(T_{j}\right)$,
$W_{p}\left(T_{j}\right)=\left|\right|_{l\in L}W_{p,l}\left(T_{j}\right)$,
and $W_{p}=\left|\right|_{j\in J}W_{p}\left(T_{j}\right)$, 
\end{itemize}
and finally $v_{p}=\left[V_{p}^{\top},F_{p}^{\top},O_{p}^{\top},W_{p}^{\top}\right]^{\top}$.

Similarly, we define 
\begin{itemize}
\item electricity price vectors $\Pi\left(T_{j}\right)=\left|\right|_{i\in I_{j}}\Pi\left(t_{i},T_{j}\right)$,
and $\Pi=\left|\right|_{j\in J}e^{-\hat{r}T_{j}}\Pi\left(T_{j}\right)$,
where $\hat{r}\in\mathbb{R}$ is a constant interest rate, 
\item fuel price vectors $G\left(t_{i},T_{j}\right)=\left|\right|_{l\in L}G_{l}\left(t_{i},T_{j}\right)$,
$G\left(T_{j}\right)=\left|\right|_{i\in I_{j}}G\left(t_{i},T_{j}\right)$,
and\linebreak{}
$G=\left|\right|_{j\in J}e^{-\hat{r}T_{j}}G\left(T_{j}\right)$, 
\item emission price vector $G_{em}\left(T_{j}\right)=\left|\right|_{i\in I_{j}}G_{em}\left(t_{i},T_{j}\right)$,
and $G_{em}=\left|\right|_{j\in J}e^{-\hat{r}T_{j}}G_{em}\left(T_{j}\right)$, 
\end{itemize}
and finally $\pi_{p}=\left[\Pi^{\top},G^{\top},G_{em}^{\top},\underset{\dim\left(W_{p}\right)}{\underbrace{0,...,0}}\right]^{\top}$,
where number of zeros matches the dimension of vector $W_{p}$. 

The profit $P_{p}\left(v_{p},\pi_{p}\right)$ of producer $p\in P$
can be calculated as
\begin{equation}
P_{p}\left(v_{p},\pi_{p}\right)=\sum_{j\in J}e^{-\hat{r}T_{j}}\left(\sum_{i\in I_{j}}P_{p}^{t_{i},T_{j}}\left(v_{p},\pi_{p}\right)\right)
\end{equation}
 where the profit $P_{p}^{t_{i},T_{j}}\left(v_{p},\pi_{p}\right)$
for each $i\in I_{j}$ and $j\in J$ can be calculated as

\[
P_{p}^{t_{i},T_{j}}\left(v_{p},\pi_{p}\right)=-\Pi\left(t_{i},T_{j}\right)V_{p}\left(t_{i},T_{j}\right)-O_{p}\left(t_{i},T_{j}\right)G_{em}\left(t_{i},T_{j}\right)-\sum_{l\in L}G_{l}\left(t_{i},T_{j}\right)F_{p,l}\left(t_{i},T_{j}\right).
\]
Ramping up and down constraints for each $j\in\left\{ 1,...,T'-1\right\} $,
where $T'$ denotes the last delivery period, $l\in L$ and $r\in R^{p,l}$
can be expressed as 
\begin{equation}
\triangle\overline{W}_{min}^{p,l,r}\leq W_{p,l,r}\left(T_{j+1}\right)-W_{p,l,r}\left(T_{j}\right)\leq\triangle\overline{W}_{max}^{p,l,r}.\label{eq:pb0}
\end{equation}
For each power plant $r\in R^{p,l}$, $\triangle\overline{W}_{max}^{p,l,r}$
and $\triangle\overline{W}_{min}^{p,l,r}$ represent maximum rates
for ramping up and down, respectively. Similarly, $\overline{W}_{max}^{p,l,r}$
denotes the maximum production and thus we can write the capacity
constraints for each power plant $r\in R^{p,l}$ as
\begin{equation}
0\leq W_{p,l,r}\left(T_{j}\right)\leq\overline{W}_{max}^{p,l,r}.\label{eq:pb1}
\end{equation}
Additionally, for each $j\in J$ the electricity sold in the forward
market must equal the actually produced electricity, i.e.
\begin{equation}
-\sum_{i\in I_{j}}V_{p}\left(t_{i},T_{j}\right)=\sum_{l\in L}\sum_{r\in R^{p,l}}W_{p,l,r}\left(T_{j}\right)\label{eq:pb2-1}
\end{equation}
and sufficient amount of fuel $l\in L$ must have been bought for
each $j\in J$, i.e.
\begin{equation}
\sum_{r\in R^{p,l}}W_{p,l,r}\left(T_{j}\right)c^{p,l,r}=\sum_{i\in I_{j}}F_{p,l}\left(t_{i},T_{j}\right)\label{eq:pb3-1}
\end{equation}
where $c^{p,l,r}>0$ is the efficiency of power plant $r$. 

The carbon emission obligation constraint can be written as
\begin{equation}
\sum_{j\in J}\sum_{i\in I_{j}}O\left(t_{i},T_{j}\right)=\sum_{j\in J}\sum_{l\in L}\sum_{r\in R^{p,l}}W_{p,l,r}\left(T_{j}\right)g^{l},\label{eq:pb4}
\end{equation}
where $g^{l}>0$ denotes the carbon emission intensity factor for
fuel $l\in L$. This constraint ensures that enough emission certificates
have been bought to cover the electricity production over the whole
planning horizon. 

We bound the number of fuel forward contracts traded for each $i\in I_{j}$,
$j\in J$, and $l\in L$ as
\begin{equation}
-F_{trade}\leq F_{p,l}\left(t_{i},T_{j}\right)\leq F_{trade}\label{eq:pb3}
\end{equation}
for some large $F_{trade}>0$ and the number of emission forward contracts
traded as 
\begin{equation}
-F_{trade}\leq O\left(t_{i},T_{j}\right)\leq F_{trade}.\label{eq:pbb4}
\end{equation}
Additionally, we also bound the number of electricity forward contracts
traded for each $i\in I_{j}$, and $j\in J$ as 
\begin{equation}
-V_{trade}\leq V_{p}\left(t_{i},T_{j}\right)\leq V_{trade}\label{eq:pb2}
\end{equation}
for some large $V_{trade}>0$. Boundedness of the number of the electricity
forward contracts traded enables us to use some of game theoretic
theory that only applies to compact sets. If $V_{trade}$ is set correctly
(see Corollary \ref{prop:bounded_volumes}), then the optimal $V_{p}\left(t_{i},T_{j}\right)$
can never lie on the boundary and hence the constraint does not have
any impact on the optimal solution.

Producers would like to maximize their profit subject to a risk budget.
Under a mean-variance optimization framework they are interested in
the mean-variance utility

\[
\begin{array}{rcl}
\Psi_{p}\left(v_{p}\right) & = & \mathbb{E}^{\mathbb{P}}\left[P_{p}\left(v_{p},\pi_{p}\right)\right]-\frac{\lambda_{p}}{2}\text{Var}^{\mathbb{P}}\left[P_{p}\left(v_{p},\pi_{p}\right)\right]\\
\\
 & = & -\mathbb{E}^{\mathbb{P}}\left[\pi_{p}\right]^{\top}v_{p}-\frac{1}{2}\lambda_{p}v_{p}^{\top}Q_{p}v_{p},
\end{array}
\]
where $\lambda_{p}>0$ is their risk preference and $Q_{p}:=\mathbb{E}^{\mathbb{P}}\left[\left(\pi_{p}-\mathbb{E}^{\mathbb{P}}\left[\pi_{p}\right]\right)\left(\pi_{p}-\mathbb{E}^{\mathbb{P}}\left[\pi_{p}\right]\right)^{\top}\right]$
an ``extended'' covariance matrix. Their objective is to solve the
following optimization problem
\begin{equation}
\begin{array}{rl}
\Phi_{p}=\underset{v_{p}}{\text{max }} & \Psi_{p}\left(v_{p}\right)\end{array}\label{eq:opt_prod}
\end{equation}
subject to (\ref{eq:pb0}), (\ref{eq:pb1}), (\ref{eq:pb2-1}), (\ref{eq:pb3-1}),
(\ref{eq:pb4}), (\ref{eq:pb3}), (\ref{eq:pbb4}), and (\ref{eq:pb2}).

\subsection{Consumer}

We make the assumption that demand is completely inelastic and that
each consumer $c\in C$ is responsible for satisfying a proportion
$p_{c}\in\left[0,1\right]$ of the total demand $D\left(T_{j}\right)$
at time $T_{j}$, $j\in J$. Since $p_{c}$ is a proportion clearly
$\sum_{c\in C}p_{c}=1.$ 

For further argumentation let us define electricity trading vectors
$V_{c}\left(T_{j}\right)=\left|\right|_{i\in I_{j}}V_{c}\left(t_{i},T_{j}\right)$
and $V_{c}=\left|\right|_{j\in J}V_{c}\left(T_{j}\right)$. 

Consumer's profit can be calculated as
\begin{equation}
P_{c}\left(V_{c},\Pi\right)=\sum_{j\in J}e^{-\hat{r}T_{j}}\left(\sum_{i\in I_{j}}-\Pi\left(t_{i},T_{j}\right)V_{c}\left(t_{i},T_{j}\right)+s_{c}p_{c}D\left(T_{j}\right)\right),\label{eq:23}
\end{equation}
where $\hat{r}\in\mathbb{R}$ denotes a constant interest rate and
$s_{c}\in\mathbb{R}$ denotes a contractually fixed price that consumer
$c\in C$ receives for selling the electricity further to households
or businesses. Demand is expected to be satisfied for each $T_{j}$,
i.e.
\begin{equation}
\sum_{i\in I_{j}}V_{c}\left(t_{i},T_{j}\right)=p_{c}D\left(T_{j}\right).\label{eq:cb2}
\end{equation}
At the time of calculating the optimal decisions, consumers assume
that they know the future realization of demand $D\left(T_{j}\right)$
precisely. If the knowledge about the future realization of the demand
changes, then players can take recursive actions by recalculating
their optimal decisions with the updated demand forecast.

Note further that the contractually fixed price $s_{c}$ only affects
the optimal objective value of consumer $c\in C$, but not also his
optimal solution. Since we are primarily interested in optimal solutions,
we simplify the notation and set $s_{c}=0$. The correct optimal value
can always be calculated via post-processing when an optimal solution
is already known. This is sometimes needed for risk management purposes.

We bound the number of the electricity forward contracts for each
$i\in I_{j}$ and $j\in J$ a consumer is allowed to trade as 
\begin{equation}
-V_{trade}\leq V_{c}\left(t_{i},T_{j}\right)\leq V_{trade}\label{eq:cb1}
\end{equation}
for some large $V_{trade}>0$ for similar reasons as for producers.
We will determine the appropriate $V_{trade}$ in Corollary \ref{prop:bounded_volumes}.

Consumers would like to maximize their profit subject to a risk budget.
Under a mean-variance optimization framework they are interested in
the mean-variance utility

\[
\begin{array}{rcl}
\Psi_{c}\left(V_{c}\right) & = & \mathbb{E}^{\mathbb{P}}\left[P_{c}\left(V_{c},\Pi\right)\right]-\frac{\lambda_{p}}{2}\text{Var}^{\mathbb{P}}\left[P_{c}\left(V_{c},\Pi\right)\right]\\
\\
 & = & -\mathbb{E}^{\mathbb{P}}\left[\Pi\right]^{\top}V_{c}-\frac{\lambda}{2}V_{c}^{\top}Q_{c}V_{c},
\end{array}
\]
where $\lambda_{c}>0$ is their risk preference and $Q_{c}:=\mathbb{E}^{\mathbb{P}}\left[\left(\Pi-\mathbb{E}^{\mathbb{P}}\left[\Pi\right]\right)\left(\Pi-\mathbb{E}^{\mathbb{P}}\left[\Pi\right]\right)^{\top}\right]$
a covariance matrix. Their objective is to solve the following optimization
problem
\begin{equation}
\Phi_{c}=\underset{V_{c}}{\text{max }}\Psi_{c}\left(V_{c}\right)\label{eq:opt_con}
\end{equation}
subject to (\ref{eq:cb2}) and (\ref{eq:cb1}).

\subsection{The hypothetical market agent}

Given the electricity $\Pi$, fuel $G$, and emission $G_{em}$ price
vectors, each producer $p\in P$ and each consumer $c\in C$ can calculate
their optimal electricity trading vectors $V_{p}$ and $V_{c}$ by
solving (\ref{eq:opt_prod}) and (\ref{eq:opt_con}), respectively.
However, the players are not necessary able to execute their calculated
optimal trading strategies because they may not find the counterparty
to trade with. In reality each contract consists of a buyer and a
seller, which imposes an additional constraint (also called the market
clearing constraint) that matches the number of short and long electricity
forward contracts for each $i\in I_{j}$ and $j\in J$ as
\begin{equation}
\sum_{c\in C}V_{c}\left(t_{i},T_{j}\right)+\sum_{p\in P}V_{p}\left(t_{i},T_{j}\right)=0.\label{eq:mk}
\end{equation}
The electricity market is responsible for satisfying this constraint
by matching buyers with sellers. The matching is done through sharing
of the price and order book information among all market participants.
If at the current price there are more long contract than short contracts,
it means that the current price is too low and asks will start to
be submitted at higher prices. The converse occurs, if there are more
short contracts than long contracts. Eventually, the electricity price
at which the number of long and short contracts matches is found.
At such a price, (\ref{eq:mk}) is satisfied ``naturally'' without
explicitly requiring the players to satisfy it. They do so because
it is in their best interest, i.e. it maximizes their mean-variance
objective functions.

The question is how to formulate such a ``natural'' constraint in
an optimization framework. A naive approach of writing the market
clearing constraint as an ordinary constraint forces the players to
satisfy it regardless of the price. We need a mechanism that models
the matching of buyers and sellers as it is performed by the electricity
market. For this purpose, we introduce a hypothetical market agent
that is allowed to slowly change electricity prices to ensure that
(\ref{eq:mk}) is satisfied.

Assume first that the feasible set of the hypothetical agent must
satisfy

\begin{equation}
-\Pi_{max}\leq\mathbb{E}^{\mathbb{P}}\left[\Pi\left(t_{i},T_{j}\right)\right]\leq\Pi_{max}\label{eq:m1}
\end{equation}
for all $i\in I_{j}$, $j\in J$, and for some large $\Pi_{max}>0$.
We write (\ref{eq:m1}) in matrix notation as
\begin{equation}
B_{M}\mathbb{E}^{\mathbb{P}}\left[\Pi\right]\leq b_{M}\label{eq:m1mat}
\end{equation}
for some $B_{M}\in\mathbb{R}^{2N\times N}$ and $b_{M}\in\mathbb{R}^{2N}$.
Boundedness of prices is a reasonable assumption, because unbounded
prices lead to infinite cash flows, which must result in the bankruptcy
of one of the counterparties involved. Besides limiting the possibility
of the bankruptcy of the market players, the boundedness of prices
also allows us to use game theoretic results that only apply to compact
sets. If $\Pi_{max}$ is set large enough (see Lemma \ref{prop:bounded_prices}),
then the optimal $\mathbb{E}^{\mathbb{P}}\left[\Pi\right]$ can never
be on the boundary of the feasible region. Hence the constraint does
not have any practical impact on the equilibrium electricity price,
but, as we will see later, it significantly simplifies the theoretical
analysis.

Let the hypothetical market agent have the following profit function
\begin{equation}
P_{M}\left(\Pi,V\right)=\sum_{j\in J}e^{-\hat{r}T_{j}}\left[\sum_{i\in I_{j}}\Pi\left(t_{i},T_{j}\right)\left(\sum_{c\in C}V_{c}\left(t_{i},T_{j}\right)+\sum_{p\in P}V_{p}\left(t_{i},T_{j}\right)\right)\right]\label{eq:market-1}
\end{equation}
and the expected profit
\begin{equation}
\Psi_{M}\left(\mathbb{E}^{\mathbb{P}}\left[\Pi\right],V\right)=\mathbb{E}^{\mathbb{P}}\left[P_{M}\left(V,\Pi\right)\right],\label{eq:utility_market}
\end{equation}
where $V=\left[V_{P}^{\top},V_{C}^{\top}\right]^{\top}$, $V_{P}=\left|\right|_{p\in P}V_{p}$,
and $V_{C}=\left|\right|_{c\in C}V_{c}$ and let the hypothetical
market agent attempt to solve
\begin{equation}
\Phi_{M}\left(V\right)=\underset{\mathbb{E}^{\mathbb{P}}\left[\Pi\right]}{\text{max }}\Psi_{M}\left(\mathbb{E}^{\mathbb{P}}\left[\Pi\right],V\right)\label{eq:SO}
\end{equation}
subject to (\ref{eq:m1mat}). The KKT conditions for (\ref{eq:SO})
in the matrix notation read
\begin{equation}
\begin{array}{rcl}
\sum_{c\in C}V_{c}+\sum_{p\in P}V_{p}-B_{M}^{\top}\eta_{M} & = & 0\\
\\
\left(B_{M}\mathbb{E}^{\mathbb{P}}\left[\Pi\right]-b_{M}\right)^{\top}\eta_{M} & = & 0\\
\\
B_{M}\mathbb{E}^{\mathbb{P}}\left[\Pi\right]-b_{M} & \leq & 0\\
\\
\eta_{M} & \geq & 0,
\end{array}\label{eq:KKT_hyp_mark}
\end{equation}
where $\eta_{M}$ denotes the dual variables of (\ref{eq:m1mat}).
Since the optimal $\mathbb{E}^{\mathbb{P}}\left[\Pi\right]$ can never
be on the boundary of the feasible region (see Lemma \ref{prop:bounded_prices}),
we can conclude that $\eta_{M}=0$. Optimality conditions (\ref{eq:KKT_hyp_mark})
are then simplified to (\ref{eq:mk}). Therefore, we can conclude
that (\ref{eq:mk}) and (\ref{eq:SO}) are equivalent, if $\eta_{M}=0$
truly holds. The relationship between (\ref{eq:mk}) and (\ref{eq:SO})
will be investigated rigorously in Proposition \ref{prop:NE-CErelation}.

Note, that the equivalence of (\ref{eq:mk}) and (\ref{eq:SO}) is
a theoretical result that has to be applied with caution in an algorithmic
framework. Formulation (\ref{eq:SO}) is clearly unstable since only
a small mismatch in the market clearing constraint sends the prices
to $\pm\Pi_{max}$. A description of an algorithm for the computation
of the equilibrium electricity price exceeds the scope of this paper
and will be examined separately.

For the further argumentation we define $v_{P}=\left|\right|_{p\in P}v_{p}$
and $v=\left[v_{P}^{\top},V_{C}^{\top}\right]^{\top}$.

\subsection{\label{sub:Construction-of-price}Construction of the equilibrium
electricity price process}

In this subsection we take a closer look at the electricity price
process $\Pi\left(t_{i},T_{j}\right)_{i\in I_{j}}$, where the expectation
of the process is defined internally as a decision of the hypothetical
market agent in order to match supply and demand. 

Since electricity prices $\Pi\left(t_{i},T_{j}\right)_{i\in I_{j}}$
are adapted for any fixed $j\in J$ and $\left|\mathbb{E}^{\mathbb{P}}\left[\Pi\right]\right|<\infty$
(see Constraint (\ref{eq:m1})), they can be uniquely decomposed,
using the Doob Decomposition Theorem, into a sum of a martingale process
$M\left(t_{i},T_{j}\right)_{i\in I_{j}}$ and an integrable predictable
process $A\left(t_{i},T_{j}\right)_{i\in I_{j}}$, $A\left(t_{0},T_{j}\right)=0$,
such that
\[
\Pi\left(t_{i},T_{j}\right)=A\left(t_{i},T_{j}\right)+M\left(t_{i},T_{j}\right)
\]
for every $i\in I_{j}$. Define  
\[
M\left(t_{i},T_{j}\right):=\Pi\left(t_{0},T_{j}\right)+\sum_{k=1}^{i}\left(\Pi\left(t_{k},T_{j}\right)-\mathbb{E}^{\mathbb{P}}\left[\Pi\left(t_{k},T_{j}\right)\left|\mathcal{F}_{k-1}\right.\right]\right)
\]
and
\[
A\left(t_{i},T_{j}\right):=\sum_{k=1}^{i}\left(\mathbb{E}^{\mathbb{P}}\left[\Pi\left(t_{k},T_{j}\right)\left|\mathcal{F}_{k-1}\right.\right]-\Pi\left(t_{k-1},T_{j}\right)\right),
\]
where $\mathcal{F}_{k}:=\mathcal{F}_{t_{k}}$. It is easy to see that
$M\left(t_{i},T_{j}\right)$ is a martingale since 
\[
\mathbb{E}^{\mathbb{P}}\left[M\left(t_{i},T_{j}\right)-M\left(t_{i-1},T_{j}\right)\left|\mathcal{F}_{i-1}\right.\right]=0\quad\text{a.s.}
\]
Moreover, $A\left(t_{i},T_{j}\right)$ is predictable, that is, $\mathcal{F}_{i-1}$-measurable.

Note that 
\[
\mathbb{E}^{\mathbb{P}}\left[\Pi\left(t_{i},T_{j}\right)\left|\mathcal{F}_{i-1}\right.\right]=A\left(t_{i},T_{j}\right)+M\left(t_{i-1},T_{j}\right),
\]
and 
\begin{equation}
\mathbb{E}^{\mathbb{P}}\left[\Pi\left(t_{i},T_{j}\right)\left|\mathcal{F}_{0}\right.\right]=\mathbb{E}^{\mathbb{P}}\left[A\left(t_{i},T_{j}\right)\left|\mathcal{F}_{0}\right.\right]+M\left(t_{0},T_{j}\right).\label{eq:fil0}
\end{equation}
Let us now allow the hypothetical market agent to choose 
\[
\tilde{A_{i}}\left(T_{j}\right):=\left[\tilde{A}\left(t_{i+1},T_{j}\right),\mathbb{E}^{\mathbb{P}}\left[\tilde{A}\left(t_{i+2},T_{j}\right)\left|\mathcal{F}_{i}\right.\right],...,\mathbb{E}^{\mathbb{P}}\left[\tilde{A}\left(t_{\max\left\{ I_{j}\right\} },T_{j}\right)\left|\mathcal{F}_{i}\right.\right]\right]^{\top}
\]
at time $t_{i}$, $i\in I_{j}$ for all $j\in J$. We then model $\Pi\left(t_{i},T_{j}\right)$
with a new probability measure $\tilde{\mathbb{P}}$ where $A\left(t_{i},T_{j}\right)$
is defined internally and not by Doob decomposition of $\Pi\left(t_{i},T_{j}\right)$
as before. More formally, we define a new probability measure $\tilde{\mathbb{P}}:\mathcal{F}\times\mathbb{R}^{N}\rightarrow\left[0,1\right]$
such that for any fixed $j\in J$, $i\in I_{j}$, and for all $D\in\mathcal{F}_{i-1}$
\[
\begin{array}{rcl}
\tilde{\mathbb{P}}\left(\Pi\left(t_{i},T_{j}\right)\in D;\tilde{A}_{i-1}\left(T_{j}\right)\right) & = & \mathbb{P}\left(M\left(t_{i},T_{j}\right)+\tilde{A}\left(t_{i},T_{j}\right)\in D\right)\\
\\
 & = & \mathbb{P}\left(M\left(t_{i},T_{j}\right)+A\left(t_{i},T_{j}\right)+\tilde{A}\left(t_{i},T_{j}\right)-A\left(t_{i},T_{j}\right)\in D\right)\\
\\
 & = & \mathbb{P}\left(\Pi\left(t_{i},T_{j}\right)+\tilde{A}\left(t_{i},T_{j}\right)-A\left(t_{i},T_{j}\right)\in D\right)\\
\\
 & = & \mathbb{P}\left(\Pi\left(t_{i},T_{j}\right)\in\varphi\left(D,\tilde{A}\left(t_{i},T_{j}\right)-A\left(t_{i},T_{j}\right)\right)\right)
\end{array}
\]
where $\varphi:\mathcal{F}\times\mathbb{R}\rightarrow\mathcal{F}$
denotes a translation of a set, i.e. $\varphi\left(D,\Delta d\right):=\left\{ d:d+\Delta d\in D\right\} $.
Since
\begin{equation}
\mathbb{E}^{\tilde{\mathbb{P}}}\left[\Pi\left(t_{k},T_{j}\right)\left|\mathcal{F}_{i}\right.;\tilde{A_{i}}\left(T_{j}\right)\right]=\mathbb{E}^{\mathbb{P}}\left[\tilde{A}\left(t_{k},T_{j}\right)\left|\mathcal{F}_{i}\right.\right]+M\left(t_{i},T_{j}\right)\label{eq:con}
\end{equation}
where $k\in\left\{ i+1,...,\max\left\{ I_{j}\right\} \right\} $,
this selection can be interpreted as determining the term structure
of the expected power price relative to the current value of $M\left(t_{i},T_{j}\right)$. 

At time $t_{0}$, when players calculate their optimal decisions for
the first time, they assume that they will execute their strategies
without any future alterations. However, we allow recourse at a later
step but this is not taken into account at time $t_{0}$. Therefore,
\begin{equation}
\tilde{A}\left(t_{i},T_{j}\right)=\mathbb{E}^{\mathbb{P}}\left[\tilde{A}\left(t_{i},T_{j}\right)\left|\mathcal{F}_{k}\right.\right]\label{eq:eqv}
\end{equation}
for all $k\in\left\{ 0,...,i-1\right\} $, $j\in J$, and $i\in I_{j}$.
In such a setting, it is enough to determine $\tilde{A}_{0}\left(T_{j}\right)$,
since all other expected prices can be derived from these using (\ref{eq:con})
and (\ref{eq:eqv}). 

The variance under the new probability measure $\tilde{\mathbb{P}}$
can be calculated as 

\begin{equation}
\begin{array}{rcl}
\text{Var}^{\tilde{\mathbb{P}}}\left(\Pi\left(t_{i},T_{j}\right)\left|\mathcal{F}_{0}\right.;\tilde{A_{i}}\left(T_{j}\right)\right) & = & \text{Var}^{\mathbb{P}}\left(\tilde{A}\left(t_{i},T_{j}\right)+M\left(t_{i},T_{j}\right)\left|\mathcal{F}_{0}\right.\right)\\
\\
 & = & \text{Var}^{\mathbb{P}}\left(M\left(t_{i},T_{j}\right)\left|\mathcal{F}_{0}\right.\right).
\end{array}\label{eq:var}
\end{equation}
We can see that the variance depends only on the process $M\left(t_{i},T_{j}\right)$
and can not be influenced by the hypothetical market agent. Using
a similar reasoning as in (\ref{eq:var}), we can also conclude 
\begin{equation}
\mathbb{E}^{\tilde{\mathbb{P}}}\left[\left(\pi_{p}-\mathbb{E}^{\tilde{\mathbb{P}}}\left[\pi_{p};\tilde{A}_{0}\right]\right)\left(\pi_{p}-\mathbb{E}^{\tilde{\mathbb{P}}}\left[\pi_{p};\tilde{A}_{0}\right]\right)^{\top};\tilde{A}_{0}\right]=Q_{p}\label{eq:prodQ}
\end{equation}
and 
\begin{equation}
\mathbb{E}^{\tilde{\mathbb{P}}}\left[\left(\Pi-\mathbb{E}^{\tilde{\mathbb{P}}}\left[\Pi;\tilde{A}_{0}\right]\right)\left(\Pi-\mathbb{E}^{\tilde{\mathbb{P}}}\left[\Pi;\tilde{A}_{0}\right]\right)^{\top};\tilde{A}_{0}\right]=Q_{c},\label{eq:conQ}
\end{equation}
where $\tilde{A}_{0}=\left|\right|_{j\in J}\tilde{A}_{0}\left(T_{j}\right)$.

Without loss of generality, we may set $M\left(t_{0},T_{j}\right)=0$.
Then from (\ref{eq:con}) and (\ref{eq:eqv}),
\begin{equation}
\begin{array}{rcl}
\mathbb{E}^{\tilde{\mathbb{P}}}\left[\Pi\left(t_{i},T_{j}\right)\left|\mathcal{F}_{0}\right.;\tilde{A_{i}}\left(T_{j}\right)\right] & = & \mathbb{E}^{\mathbb{P}}\left[\tilde{A}\left(t_{i},T_{j}\right)\left|\mathcal{F}_{0}\right.\right]\\
\\
 & = & \tilde{A}\left(t_{i},T_{j}\right)
\end{array}
\end{equation}
 for all $i\in I_{j}$ and $j\in J$. Allowing the hypothetical market
agent to choose $\tilde{A}\left(t_{i},T_{j}\right)$ is thus the same
as allowing it to choose $\mathbb{E}^{\tilde{\mathbb{P}}}\left[\Pi\left(t_{i},T_{j}\right)\left|\mathcal{F}_{0}\right.;\tilde{A}\left(t_{i},T_{j}\right)\right]$.
In the rest of the paper, we simplify the notation by writing $\mathbb{E}^{\tilde{\mathbb{P}}}\left[\Pi\left(t_{i},T_{j}\right)\right]$
when we actually mean $\mathbb{E}^{\tilde{\mathbb{P}}}\left[\Pi\left(t_{i},T_{j}\right)\left|\mathcal{F}_{0}\right.;\tilde{A}\left(t_{i},T_{j}\right)\right]$. 

The measure $\tilde{\mathbb{P}}$ corresponds to a physical measure
and $\mathbb{E}^{\tilde{\mathbb{P}}}\left[\cdot\right]$ corresponds
to a real world expectation. The hypothetical market agent is allowed
to choose $\tilde{A}\left(t_{i},T_{j}\right)$ and consequently the
physical measure $\tilde{\mathbb{P}}$. The aim of the hypothetical
market agent could be interpreted as finding the physical measure
$\tilde{\mathbb{P}}$ that is consistent with the facts (e.g. fuel
forward prices, emission prices, demand forecasts etc.) observable
in the real world.

\subsection{Matrix notation}

The analysis of the problem is greatly simplified if a more compact
notation is introduced. In this subsection we will also rewrite equations
using the new probability measure $\tilde{\mathbb{P}}$ instead of
$\mathbb{P}$ where applicable.

The profit of producer $p\in P$ can be written as
\[
P_{p}\left(v_{p},\pi_{p}\right)=-\pi_{p}^{\top}v_{p}.
\]
The equality constraints can be expressed as
\[
A_{p}v_{p}=0
\]
and the inequality constraints as 
\[
B_{p}v_{p}\leq b_{p}
\]
for some $A_{p}\in\mathbb{R}^{\left|J\right|\left(\left|L\right|+1\right)+1\times\dim v_{p}}$,
$B_{p}\in\mathbb{R}^{n_{p}\times\dim v_{p}}$ and $b_{p}\in\mathbb{R}^{n_{p}}$,
where $n_{p}$ denotes the number of all inequality constraints of
producer $p\in P$. Define a feasible set 
\[
S_{p}:=\left\{ v_{p}:A_{p}v_{p}=a_{p}\text{ \text{and }}B_{p}v_{p}\leq b_{p}\right\} .
\]
It is useful to investigate the inner structure of the matrices. By
considering equality constraints (\ref{eq:pb2-1}), (\ref{eq:pb3-1}),
and (\ref{eq:pb4}) we can see that 
\begin{equation}
A_{p}=\left[\begin{array}{ccc}
\hat{A}_{1} & 0 & \hat{A}_{3,p}\\
0 & \hat{A}_{2} & \hat{A}_{4,p}
\end{array}\right]\label{eq:Ap}
\end{equation}
where $\hat{A}_{1}\in\mathbb{R}^{\left|J\right|\times N},\hat{A}_{2}\in\mathbb{R}^{\left(\left|J\right|\left|L\right|+1\right)\times N\left(\left|L\right|+1\right)},\hat{A}_{3,p}\in\mathbb{R}^{\left|J\right|\times\dim W_{p}},\hat{A}_{4,p}\in\mathbb{R}^{\left(\left|J\right|\left|L\right|+1\right)\times\dim W_{p}}$.
One can see that matrices $\hat{A}_{1}$ and $\hat{A}_{2}$ are independent
of producer $p\in P$ and matrices $\hat{A}_{3,p}$ and $\hat{A}_{4,p}$
depend on producer $p\in P$. One can further investigate the structure
of $\hat{A}_{1}$ and see 
\begin{equation}
\hat{A}_{1}=\left[\begin{array}{ccc}
1_{1} &  & 0\\
 & \ddots\\
0 &  & 1_{T'}
\end{array}\right],\label{eq:A1}
\end{equation}
where $1_{j}$, $j\in J$ is a row vector of ones of length $\left|I_{j}\right|$.
Similarly,
\begin{equation}
\hat{A}_{2}=\left[\begin{array}{cccc}
\hat{A}_{1} & \cdots & 0 & 0\\
\vdots & \ddots & \vdots & \vdots\\
0 & \cdots & \hat{A}_{1} & 0\\
0 & \cdots & 0 & 1_{\left|N\right|}
\end{array}\right],\label{eq:A2}
\end{equation}
where the number of rows in the block notation above is $\left|L\right|+1$.
The first $\left|L\right|$ rows correspond to (\ref{eq:pb3-1}) and
the last row correspond to (\ref{eq:pb4}). 

In a compact notation, the mean-variance utility of producer $p\in P$
can be calculated as
\[
\begin{array}{rcl}
\Psi_{p}\left(v_{p},\mathbb{E}^{\mathbb{\tilde{P}}}\left[\Pi\right]\right) & = & \mathbb{E}^{\tilde{\mathbb{P}}}\left[-\pi_{p}^{\top}v_{p}-\frac{1}{2}\lambda_{p}v_{p}^{\top}\left(\pi_{p}-\mathbb{E}^{\tilde{\mathbb{P}}}\left[\pi_{p}\right]\right)\left(\pi_{p}-\mathbb{E}^{\tilde{\mathbb{P}}}\left[\pi_{p}\right]\right)^{\top}v_{p}\right]\\
\\
 & = & -\mathbb{E}^{\tilde{\mathbb{P}}}\left[\pi_{p}\right]^{\top}v_{p}-\frac{1}{2}\lambda_{p}v_{p}^{\top}Q_{p}v_{p},
\end{array}
\]
where (\ref{eq:prodQ}) was used. The inner structure of matrix $Q_{p}$
is the following 
\begin{equation}
Q_{p}=\left[\begin{array}{ccc}
\hat{Q}_{1} & \hat{Q}_{2} & 0\\
\hat{Q}_{2}^{\top} & \hat{Q}_{3} & 0\\
0 & 0 & 0
\end{array}\right],
\end{equation}
where $\hat{Q}_{1}\in\mathbb{R}^{N\times N},\hat{Q}_{2}\in\mathbb{R}^{N\times\left(\dim B_{p}+\dim O_{p}\right)}=\mathbb{R}^{N\times N\left(\left|L\right|+1\right)},\hat{Q}_{3}\in\mathbb{R}^{N\left(\left|L\right|+1\right)\times N\left(\left|L\right|+1\right)}$.
One can see that $\hat{Q}_{1}$, $\hat{Q}_{2}$, and $\hat{Q}_{3}$
do not depend on producer $p\in P$. The size of the larger matrix
$Q_{p}$ depends on producer $p\in P$, because different producers
have different number of power plants. 

Producer $p\in P$ attempts to solve the following optimization problem
\[
\Phi_{p}\left(\mathbb{E}^{\tilde{\mathbb{P}}}\left[\Pi\right]\right)=\underset{v_{p}\in S_{p}}{\text{max }}-\mathbb{E}^{\tilde{\mathbb{P}}}\left[\pi_{p}\right]^{\top}v_{p}-\frac{1}{2}\lambda_{p}v_{p}^{\top}Q_{p}v_{p}.
\]

The profit of consumer $c\in C$ can be written as
\[
P_{c}\left(V_{c},\Pi\right)=-\Pi^{\top}V_{c}.
\]
Note here that we set $s_{c}=0$, w.l.o.g. The equality constraints
can be expressed as
\[
A_{c}V_{c}=a_{c}
\]
and the inequality constraints as 
\[
B_{c}V_{c}\leq b_{c}
\]
where $A_{c}=\hat{A}_{1}$, $B_{c}\in\mathbb{R}^{2N\times N}$, $a_{c}\in\mathbb{R}^{\left|J\right|}$
and $b_{c}\in\mathbb{R}^{N}$. Define a feasible set
\[
S_{c}:=\left\{ V_{c}\in\mathbb{R}^{N}:A_{c}V_{c}=a_{c}\text{ \text{and }}B_{c}V_{c}\leq b_{c}\right\} .
\]

In a compact notation, the mean-variance utility of a consumer $c\in C$
can be calculated as
\[
\begin{array}{rcl}
\Psi_{c}\left(V_{c},\mathbb{E}^{\tilde{\mathbb{P}}}\left[\Pi\right]\right) & = & \mathbb{E}^{\tilde{\mathbb{P}}}\left[-\Pi^{\top}V_{c}-\frac{1}{2}\lambda_{c}V_{c}^{\top}\left(\Pi-\mathbb{E}^{\tilde{\mathbb{P}}}\left[\Pi\right]\right)\left(\Pi-\mathbb{E}^{\tilde{\mathbb{P}}}\left[\Pi\right]\right)^{\top}V_{c}\right]\\
\\
 & = & -\mathbb{E}^{\tilde{\mathbb{P}}}\left[\Pi\right]^{\top}V_{c}-\frac{\lambda}{2}V_{c}^{\top}Q_{c}V_{c},
\end{array}
\]
where (\ref{eq:conQ}) was used. Moreover, note that $Q_{c}=\hat{Q}_{1}$
for all $c\in C$. 

Consumer $c\in C$ attempts to solve the following optimization problem
\[
\Phi_{c}\left(\mathbb{E}^{\tilde{\mathbb{P}}}\left[\Pi\right]\right)=\underset{V_{c}\in S_{c}}{\text{max }}-\mathbb{E}^{\tilde{\mathbb{P}}}\left[\Pi\right]^{\top}V_{c}-\frac{\lambda}{2}V_{c}^{\top}Q_{c}V_{c}.
\]

The profit function of the hypothetical market agent can be written
in a compact notation as
\begin{equation}
P_{M}\left(\Pi,V\right)=\Pi^{\top}\left(\sum_{c\in C}V_{c}+\sum_{p\in P}V_{p}\right)\label{eq:market_profit}
\end{equation}
and the expected utility as
\begin{equation}
\begin{array}{rcl}
\Psi_{M}\left(\mathbb{E}^{\tilde{\mathbb{P}}}\left[\Pi\right],V\right) & = & \mathbb{E}^{\tilde{\mathbb{P}}}\left[P_{M}\left(\Pi,V\right)\right]\\
\\
 & = & \mathbb{E}^{\tilde{\mathbb{P}}}\left[\Pi\right]^{\top}\left(\sum_{c\in C}V_{c}+\sum_{p\in P}V_{p}\right).
\end{array}\label{eq:utility_market_mx}
\end{equation}
Consequently, the hypothetical market agent's objective is to solve
\begin{equation}
\Phi_{M}\left(V\right)=\underset{\mathbb{E}^{\tilde{\mathbb{P}}}\left[\Pi\right]\in S_{M}}{\text{max }}\mathbb{E}^{\tilde{\mathbb{P}}}\left[\Pi\right]^{\top}\left(\sum_{c\in C}V_{c}+\sum_{p\in P}V_{p}\right),\label{eq:SO-1}
\end{equation}
where the feasible set of the hypothetical market agent is defined
as 
\[
S_{M}:=\left\{ \mathbb{E}^{\tilde{\mathbb{P}}}\left[\Pi\right]\in\mathbb{R}^{N}:-\Pi_{max}\leq\mathbb{E}^{\tilde{\mathbb{P}}}\left[\Pi\left(t_{i},T_{j}\right)\right]\leq\Pi_{max}\text{ \text{for all }}j\in J,i\in I_{j}\right\} .
\]

\section{Analysis of the model\label{sec:Analysis}}

In this section we analyze the existence and uniqueness of solutions
of the model defined in the previous section. 

\begin{definition}\label{Competitive-Equilibrium-(CE)}Competitive
Equilibrium (CE)

Decisions $v^{*}$ and $\mathbb{E}^{\tilde{\mathbb{P}}}\left[\Pi\right]^{*}$
constitute a competitive equilibrium if
\begin{enumerate}
\item For every producer $p\in P$, $v_{p}^{*}$ is a strategy such that
\begin{equation}
\Psi_{p}\left(v_{p},\mathbb{E}^{\tilde{\mathbb{P}}}\left[\Pi\right]^{*}\right)\leq\Psi_{p}\left(v_{p}^{*},\mathbb{E}^{\tilde{\mathbb{P}}}\left[\Pi\right]^{*}\right)\label{eq:CE_producers}
\end{equation}
for all $v_{p}\in S_{p}$;
\item For every consumer $c\in C$, $V_{c}^{*}$ is a strategy such that
\begin{equation}
\Psi_{c}\left(V_{c},\mathbb{E}^{\tilde{\mathbb{P}}}\left[\Pi\right]^{*}\right)\leq\Psi_{c}\left(V_{c}^{*},\mathbb{E}^{\tilde{\mathbb{P}}}\left[\Pi\right]^{*}\right)\label{eq:CE_consumers}
\end{equation}
for all $V_{c}\in S_{c}$;
\item For each $i\in I_{j}$ and $j\in J$ 
\begin{equation}
0=\sum_{c\in C}V_{c}\left(t_{i},T_{j}\right)+\sum_{p\in P}V_{p}\left(t_{i},T_{j}\right)\label{eq:vol}
\end{equation}
must hold.
\end{enumerate}
\end{definition}

\begin{definition}\label{Nash-Equilibrium-(NE)}Nash Equilibrium
(NE)

Decisions $v^{*}$ and $\mathbb{E}^{\tilde{\mathbb{P}}}\left[\Pi\right]^{*}$
constitute a Nash equilibrium if
\begin{enumerate}
\item For every producer $p\in P$, $v_{p}^{*}$ is a strategy such that
\begin{equation}
\Psi_{p}\left(v_{p},\mathbb{E}^{\tilde{\mathbb{P}}}\left[\Pi\right]^{*}\right)\leq\Psi_{p}\left(v_{p}^{*},\mathbb{E}^{\tilde{\mathbb{P}}}\left[\Pi\right]^{*}\right)
\end{equation}
for all $v_{p}\in S_{p}$;
\item For every consumer $c\in C$, $V_{c}^{*}$ is a strategy such that
\begin{equation}
\Psi_{c}\left(V_{c},\mathbb{E}^{\tilde{\mathbb{P}}}\left[\Pi\right]^{*}\right)\leq\Psi_{c}\left(V_{c}^{*},\mathbb{E}^{\tilde{\mathbb{P}}}\left[\Pi\right]^{*}\right)
\end{equation}
for all $V_{c}\in S_{c}$;
\item Price vector $\mathbb{E}^{\tilde{\mathbb{P}}}\left[\Pi\right]^{*}$
maximizes the objective function of the hypothetical market agent,
i.e. 
\begin{equation}
\Psi_{M}\left(\mathbb{E}^{\tilde{\mathbb{P}}}\left[\Pi\right],v^{*}\right)\leq\Psi_{M}\left(\mathbb{E}^{\tilde{\mathbb{P}}}\left[\Pi\right]^{*},v^{*}\right)\label{eq:market}
\end{equation}
for all $\mathbb{E}^{\tilde{\mathbb{P}}}\left[\Pi\right]\in S_{M}$.
\end{enumerate}
\end{definition}

\begin{assumption}\label{ass: as1} For all $p\in P$, the exists
vector $v_{p}$ such that $A_{p}v_{p}=a_{p}$ a.s. and $B_{p}v_{p}<b_{p}$
a.s., for all $c\in C$, there exists vector $V_{c}$ such that $A_{c}V_{c}=a_{c}$
a.s. and $B_{c}V_{c}<b_{c}$ a.s., and the vectors $V_{p}$ and $V_{c}$
can be chosen so that (\ref{eq:vol}) is satisfied.\end{assumption}

\subsection{Existence of Solution}

We start the analysis of the existence of solution with stating the
following well known theorem.

\begin{theorem}\label{thm:(Debreu,-Glickberg,-Fan)}(Debreu, Glickberg,
Fan) 

Consider a strategic form game $\left\langle \mathcal{N},\left(S_{i}\right)_{i\in\mathcal{N}},\left(\Psi_{i}\right)_{i\in\mathcal{N}}\right\rangle $,
where $\mathcal{N}$ denotes a set of players. Let $s_{-i}\in S_{-i}=\prod_{i\neq j}S_{j}$
denote a vector of actions for all players except player $i\in\mathcal{N}$,
and let $s_{i}\in S_{i}$ denote an action for player $i\in\mathcal{N}$.
If for each $i\in\mathcal{N}$
\begin{enumerate}
\item feasible set $S_{i}$ is non-empty, compact and convex;
\item $\Psi_{i}\left(s_{i},s_{-i}\right)$ is continuous in $s_{-i}$
\item $\Psi_{i}\left(s_{i},s_{-i}\right)$ is continuous and concave in
$s_{i}$,
\end{enumerate}
then a pure Nash equilibrium exists.

\end{theorem}

The proof for the above theorem can be found in \cite{debreu1952asocial}.
It is straightforward to apply the theorem and show that our problem
from Definition \ref{Nash-Equilibrium-(NE)} has a solution.

\begin{corollary}\label{cor:exist}Let Assumption \ref{ass: as1}
hold. Then there exists a pure NE for Problem \ref{Nash-Equilibrium-(NE)}.
\end{corollary}

\begin{proof} A set of players $\mathcal{N}$ is composed of all
producers $p\in P$, consumers $c\in C$, and the hypothetical market
agent. Due to Assumption \ref{ass: as1} it is clear that the feasible
region of each player is non-empty. Since all constraints are affine
functions with non-strict inequalities, it is clearly also convex
and closed. Boundedness is guaranteed by (\ref{eq:pb0}), (\ref{eq:pb1}),
(\ref{eq:pb3}), (\ref{eq:pbb4}), and (\ref{eq:pb2}) for producers,
by (\ref{eq:cb1}) for consumers, and by (\ref{eq:m1}) for the hypothetical
market agent. The covariance matrix $Q_{p}$ in $\Psi_{p}\left(v_{p},\mathbb{E}^{\tilde{\mathbb{P}}}\left[\Pi\right]\right)$,
$v_{p}\in S_{p}$, $\mathbb{E}^{\tilde{\mathbb{P}}}\left[\Pi\right]\in S_{M}$
is by definition positive semidefinite for all $p\in P$ and thus
$\Psi_{p}\left(v_{p},\mathbb{E}^{\tilde{\mathbb{P}}}\left[\Pi\right]\right)$
is concave in $v_{p}$. Similarly, we can also see that the covariance
matrix $Q_{c}$ is positive semidefinite and thus $\Psi_{c}\left(V_{c},\mathbb{E}^{\tilde{\mathbb{P}}}\left[\Pi\right]\right)$,
$V_{c}\in S_{c}$, $\mathbb{E}^{\tilde{\mathbb{P}}}\left[\Pi\right]\in S_{M}$
for all $c\in C$ is concave in $V_{c}$. The function $\Psi_{M}\left(\mathbb{E}^{\tilde{\mathbb{P}}}\left[\Pi\right],v\right)$
that corresponds to the expected utility of the hypothetical market
agent is linear in $\mathbb{E}^{\tilde{\mathbb{P}}}\left[\Pi\right]$
and thus concave. Utility functions of producers, consumers, and the
hypothetical market agent are all quadratic and therefore continuous.
Thus, a pure NE exists. \end{proof}

\subsection{Uniqueness of Solution}

\begin{definition} A power plant $r\in R^{p,l}$, $p\in P$, $l\in L$
is at the upper bound at delivery time $T_{j}$, $j\in J$, if for
every $\epsilon>0$, $W_{p,l,r}\left(T_{j}\right)+\epsilon$ is infeasible.
Similarly, a power plant $r\in R^{p,l}$, $p\in P$, $l\in L$, $r\in R^{p,l}$
is at the lower bound at delivery time $T_{j}$, $j\in J$, if for
every $\epsilon>0$, $W_{p,l,r}\left(T_{j}\right)-\epsilon$ is infeasible.
Bounds on $W_{p,l,r}\left(T_{j}\right)$ are defined by constraints
(\ref{eq:pb0}) and (\ref{eq:pb1}). \end{definition}

Let us start the discussion about the uniqueness of the Nash equilibrium
with a direct consequence of Assumption \ref{ass: as1}.

\begin{lemma}\label{cor:cc1}If there exists $j\in J$ such that
all power plants are at the upper bound simultaneously, then 
\[
0>\sum_{c\in C}\sum_{i\in I_{j}}V_{c}\left(t_{i},T_{j}\right)+\sum_{p\in P}\sum_{i\in I_{j}}V_{p}\left(t_{i},T_{j}\right).
\]
Similarly, if there exists $j\in J$ such that all power plants are
at the lower bound simultaneously, then 
\[
0<\sum_{c\in C}\sum_{i\in I_{j}}V_{c}\left(t_{i},T_{j}\right)+\sum_{p\in P}\sum_{i\in I_{j}}V_{p}\left(t_{i},T_{j}\right).
\]
\end{lemma}

\begin{proof}

Using (\ref{eq:pb2-1}) and summing over all producers $p\in P$ we
get 
\begin{equation}
\sum_{p\in P}\sum_{i\in I_{j}}V_{p}\left(t_{i},T_{j}\right)=-\sum_{p\in P}\sum_{l\in L}\sum_{r\in R^{p,l}}W_{p,l,r}\left(T_{j}\right).\label{eq:ml2-1}
\end{equation}
Similarly, using (\ref{eq:cb2}) and the fact that $\sum_{c\in C}p_{c}=1$,
we get
\begin{equation}
\begin{array}{rcl}
\sum_{c\in C}\sum_{i\in I_{j}}V_{c}\left(t_{i},T_{j}\right) & = & \sum_{c\in C}p_{c}D\left(T_{j}\right)\\
\\
 & = & D\left(T_{j}\right).
\end{array}\label{eq:ml1-2}
\end{equation}
Assume now that all power plants are at the upper bound simultaneously
for some $j'\in J$. Then by Assumption \ref{ass: as1}, 
\begin{equation}
D\left(T_{j'}\right)<\sum_{p\in P}\sum_{l\in L}\sum_{r\in R^{p,l}}W_{p,l,r}\left(T_{j'}\right)
\end{equation}
and therefore by (\ref{eq:ml2-1}) and (\ref{eq:ml1-2}), 
\[
0>\sum_{c\in C}\sum_{i\in I_{j}}V_{c}\left(t_{i},T_{j}\right)+\sum_{p\in P}\sum_{i\in I_{j}}V_{p}\left(t_{i},T_{j}\right).
\]
A similar argument also holds for the lower bound.\end{proof}

\begin{assumption}\label{ass: as3}None of random variables $\Pi$,
$G$, and $G_{em}$ can be written as a linear combination of the
others a.s.\end{assumption}

In the rest of this paper we assume that Assumption \ref{ass: as1}
and Assumption \ref{ass: as3} always hold. 

\begin{lemma}\label{prop:uniq_vol}The mean-variance objective functions
of producers $p\in P$ and consumers $c\in C$ are strictly concave
in $V_{p}$, $F_{p}$, and $O_{p}$ for all $p\in P$ and in $V_{c}$
for all $c\in C$. \end{lemma}

\begin{proof}The objective function of each producer $p\in P$ can
be written as 
\[
\Psi_{p}\left(v_{p},\mathbb{E}^{\tilde{\mathbb{P}}}\left[\Pi\right]\right)=-\mathbb{E}^{\tilde{\mathbb{P}}}\left[\pi_{p}\right]^{\top}v_{p}-\frac{1}{2}\lambda_{p}v_{p}^{\top}Q_{p}v_{p},
\]
where
\[
Q_{p}=\left[\begin{array}{cc}
\hat{Q} & 0\\
0 & 0
\end{array}\right],
\]
$\hat{Q}:=\left[\begin{array}{cc}
\hat{Q}_{1} & \hat{Q}_{2}\\
\hat{Q}_{2}^{\top} & \hat{Q}_{3}
\end{array}\right]$, and $\pi':=\left[\Pi^{\top},G^{\top},G_{em}^{\top}\right]^{\top}$.
Under Assumption \ref{ass: as3}, we conclude that $\hat{Q}\succ0$.
Define $v_{p}':=\left[V_{p}^{\top},F_{p}^{\top},O_{p}^{\top}\right]^{\top}$
and $v''_{p}:=W_{p}$. Then 
\[
\mathcal{D}_{v_{p'}}\Psi_{p}\left(v_{p},\mathbb{E}^{\tilde{\mathbb{P}}}\left[\Pi\right]\right)=-\mathbb{E}^{\tilde{\mathbb{P}}}\left[\pi'\right]^{\top}-\lambda_{p}\hat{Q}v_{p}'
\]
and
\[
\mathcal{D}_{v_{p''}}\Psi_{p}\left(v_{p},\mathbb{E}^{\tilde{\mathbb{P}}}\left[\Pi\right]\right)=0.
\]
The second derivative $\mathcal{D}_{v_{p'}}^{2}\Psi_{p}\left(v_{p},\mathbb{E}^{\tilde{\mathbb{P}}}\left[\Pi\right]\right)=-\lambda_{p}\hat{Q}\prec0$
since $0<\lambda_{p}<\infty$. Thus, $\Psi_{p}\left(v_{p},\mathbb{E}^{\tilde{\mathbb{P}}}\left[\Pi\right]\right)$,
$p\in P$ are strictly concave in $V_{p}$, $F_{p}$, and $O_{p}$.
The proof for consumers $c\in C$ is similar. \end{proof}

\begin{lemma}\label{prop:Optimal-fuel-trading}Optimal fuel trading
strategies $F_{p}$ and emission trading strategies $O_{p}$ are bounded
for all producers $p\in P$.%
\footnote{Hence in there exists some $M_{1}\in\mathbb{R}$ such that $\left\Vert \left[F_{p}^{\top},O_{p}^{\top}\right]^{\top}\right\Vert \leq M_{1}<\infty$
for all $p\in P$. We set $M_{1}<F_{trade}$.%
} \end{lemma}

\begin{proof}Since $\Psi_{p}\left(v_{p},\mathbb{E}^{\tilde{\mathbb{P}}}\left[\Pi\right]\right)$
is strictly concave and quadratic in $F_{p}$ and $O_{p}$ for all
$p\in P$, and fuel and emission prices have finite expectation and
variance, it is clear that $\Psi_{p}\left(v_{p},\mathbb{E}^{\tilde{\mathbb{P}}}\left[\Pi\right]\right)\rightarrow-\infty$
as $\left\Vert F_{p}\right\Vert \rightarrow\infty$ or $\left\Vert O_{p}\right\Vert \rightarrow\infty$.
This is not optimal and thus we conclude that optimal $F_{p}$ and
$O_{p}$ must be bounded. \end{proof}

\begin{lemma}\label{prop:Unbound_vol_unbound_pr}Consider optimization
problem (\ref{eq:opt_prod}) without constraints (\ref{eq:pb2}) for
producers $p\in P$, and optimization problem (\ref{eq:opt_con})
without constraints (\ref{eq:cb1}) for consumers $c\in C$. Denote
by $V_{k}\in\mathbb{R}^{N}$ a vector of optimal volumes of any player
$k\in P\cup C$ for a given vector of expected prices $\mathbb{E}^{\tilde{\mathbb{P}}}\left[\Pi\right]\in\mathbb{R}^{N}$.
Then,
\begin{enumerate}
\item if $\left\Vert \mathbb{E}^{\tilde{\mathbb{P}}}\left[\Pi\right]\right\Vert \rightarrow\infty$,
then $\left[\sum_{k\in P\cup C}V_{k}\right]^{\top}\mathbb{E}^{\tilde{\mathbb{P}}}\left[\Pi\right]\rightarrow-\infty$,
and
\item if for each delivery period $T_{j}$, $j\in J$ there exists at least
one power plant that is not at the upper or the lower boundary, then
$\left\Vert V_{k}\right\Vert \rightarrow\infty$ for all $k\in P\cup C$
simultaneously if and only if $\left\Vert \mathbb{E}^{\tilde{\mathbb{P}}}\left[\Pi\right]\right\Vert \rightarrow\infty$. 
\end{enumerate}
\end{lemma}

\begin{proof}The necessary and sufficient conditions for $v_{k}$
to be a global maximizer of $\Psi_{k}\left(v_{k},\mathbb{E}^{\tilde{\mathbb{P}}}\left[\Pi\right]\right)$
are, due to Assumption \ref{ass: as1} that implies the Slater condition,
the following 
\begin{equation}
\begin{array}{rcl}
-\mathbb{E}^{\tilde{\mathbb{P}}}\left[\pi_{k}\right]-\lambda_{k}Q_{k}v_{k}-B_{k}^{\top}\eta_{k}-A_{k}^{\top}\mu_{k} & = & 0\\
\\
\left(B_{k}v_{k}-b_{k}\right)^{\top}\eta_{k} & = & 0\\
\\
B_{k}v_{k}-b_{k} & \leq & 0\\
\\
A_{k}v_{k}-a_{k} & = & 0\\
\\
\eta_{k} & \geq & 0.
\end{array}\label{eq:KKTprod}
\end{equation}
We are only interested in decision variables $V_{k}$. By Lemma \ref{prop:Optimal-fuel-trading},
we are allowed to remove constraints (\ref{eq:pb3}) and (\ref{eq:pbb4})
without affecting the optimal solution. Then, after removing the inequality
constraints (\ref{eq:pb2}) for producers $p\in P$, and inequality
constraints (\ref{eq:cb1}) for consumers $c\in C$, there is no inequality
constraints that involve variable $V_{k}$. Moreover, there is only
one equality constraint, i.e. (\ref{eq:pb2-1}) for producers and
(\ref{eq:cb2}) for consumers, for each delivery period $j\in J$
that involve variable $V_{k}$. Considering the first equation of
(\ref{eq:KKTprod}) for each delivery period $j\in J$ separately,
and neglecting all bounded terms, we obtain the following equivalence
\begin{equation}
\mathbb{E}^{\tilde{\mathbb{P}}}\left[\Pi\left(T_{j}\right)\right]+\lambda_{k}\hat{Q}_{1}^{j}V_{k}+\mu_{k,j}1\sim0\label{eq:mt}
\end{equation}
were $\mu_{k,j}\in\mathbb{R}$ is the dual variable of the equality
constraint (\ref{eq:pb2-1}) for producers and (\ref{eq:cb2}) for
consumers, $1\in\mathbb{R}^{\left|I_{j}\right|}$ is a vector of ones,
and $\hat{Q}_{1}^{j}\in\mathbb{R}^{\left|I_{j}\right|\times N}$ contains
only those rows of $\hat{Q}_{1}$ that correspond to $\mathbb{E}^{\tilde{\mathbb{P}}}\left[\Pi\left(T_{j}\right)\right]$
in (\ref{eq:KKTprod}). In the calculation of (\ref{eq:mt}) we took
into account the construction of the equilibrium price process in
Section \ref{sub:Construction-of-price}, and use the finding that
all variances and correlations are finite and can not be affected
by the hypothetical market agent. Therefore, $\left\Vert Q_{k}\right\Vert <\infty$.
Moreover, $\left\Vert \left[F_{p}^{\top},O_{p}^{\top}\right]^{\top}\right\Vert <\infty$
due to Lemma \ref{prop:Optimal-fuel-trading}.

Assume that $\left\Vert \mathbb{E}^{\tilde{\mathbb{P}}}\left[\Pi\left(T_{j}\right)\right]\right\Vert \rightarrow\infty$.
Then by (\ref{eq:mt}), $\left\Vert \mu_{k,j}\right\Vert \rightarrow\infty$
or $\left\Vert V_{k}\right\Vert \rightarrow\infty$.
\begin{enumerate}
\item Assume first that $\left\Vert V_{k}\right\Vert <\infty$ for all $k\in P\cup C$.
Let $\mathbb{E}^{\tilde{\mathbb{P}}}\left[\Pi\left(t_{i},T_{j}\right)\right]\rightarrow\infty$
for some $i\in I_{j}$ and $j\in J$. Then $\mu_{k,j}\rightarrow-\infty$.
Since $\left\Vert V_{k}\left(T_{j}\right)\right\Vert <\infty$, it
follows from (\ref{eq:mt}) that $\mathbb{E}^{\tilde{\mathbb{P}}}\left[\Pi\left(T_{j}\right)\right]\rightarrow\infty$
componentwise and all components of $\mathbb{E}^{\tilde{\mathbb{P}}}\left[\Pi\left(T_{j}\right)\right]$
must be equal up to a constant. From (\ref{eq:pb2-1}) and the general
interpretation of the dual variables, we can see that as $\mu_{p,j}\rightarrow-\infty$,
$p\in P$, a small increase in $\sum_{l\in L}\sum_{r\in R^{p,l}}W_{p,l,r}\left(T_{j}\right)$
would infinitely improve the objective function $\Psi_{p}\left(v_{p},\mathbb{E}^{\tilde{\mathbb{P}}}\left[\Pi\right]\right)$.
As governed by (\ref{eq:pb3-1}) and (\ref{eq:pb4}), an increase
in $\sum_{l\in L}\sum_{r\in R^{p,l}}W_{p,l,r}\left(T_{j}\right)$
would also require that more fuel and emission certificates are bought.
Since the fuel and emission prices have a finite expectation and variance,
a decrease of the objective function due to the change in fuel and
emission buying strategy as $\mu_{p,j}\rightarrow-\infty$ remains
bounded. Thus, if $\mu_{p,j}\rightarrow-\infty$, the sum of production
of all power plants $\sum_{l\in L}\sum_{r\in R^{p,l}}W_{p,l,r}\left(T_{j}\right)$
of producer $p\in P$ increases as much as allowed by constraints
(\ref{eq:pb0}) and (\ref{eq:pb1}). Since all producers share the
same electricity, fuel, and emission prices, this holds for all producers
$p\in P$ simultaneously. Then by Lemma \ref{cor:cc1}, $0>\sum_{c\in C}\sum_{i\in I_{j}}V_{c}\left(t_{i},T_{j}\right)+\sum_{p\in P}\sum_{i\in I_{j}}V_{p}\left(t_{i},T_{j}\right)$
and thus $\left[\sum_{k\in P\cup C}V_{k}\left(T_{j}\right)\right]^{\top}\mathbb{E}^{\tilde{\mathbb{P}}}\left[\Pi\left(T_{j}\right)\right]\rightarrow-\infty$.
On the other hand, if $\left\Vert \mathbb{E}^{\tilde{\mathbb{P}}}\left[\Pi\left(T_{j}\right)\right]\right\Vert <\infty$
then also $\left|\left[\sum_{k\in P\cup C}V_{k}\left(T_{j}\right)\right]^{\top}\mathbb{E}^{\tilde{\mathbb{P}}}\left[\Pi\left(T_{j}\right)\right]\right|<\infty$.
Therefore, we can conclude that $\left[\sum_{k\in P\cup C}V_{k}\right]^{\top}\mathbb{E}^{\tilde{\mathbb{P}}}\left[\Pi\right]\rightarrow-\infty$. 
\item Assume now that $\left\Vert V_{k}\right\Vert \rightarrow\infty$ for
at least one $k\in P\cup C$. Then, $V_{k}^{\top}\hat{Q}_{1}V_{k}\rightarrow\infty$
since $\hat{Q}_{1}\succ0$. Assume now that $\mathbb{E}^{\tilde{\mathbb{P}}}\left[\Pi\right]^{\top}V_{k}>-\infty$.
Then $\Psi_{k}\left(v_{k},\mathbb{E}^{\mathbb{\tilde{P}}}\left[\Pi\right]\right)\rightarrow-\infty$,
which is clearly not optimal for player $k$. Thus, $\mathbb{E}^{\tilde{\mathbb{P}}}\left[\Pi\right]^{\top}V_{k}\rightarrow-\infty$,
which concludes the proof of point 1.
\end{enumerate}
We continue with the proof of point 2. Assume that for fixed $j\in J$
not all power plants are at the upper or lower bound simultaneously.
We have already seen above that if $\left|\mu_{k,j}\right|\rightarrow\infty$
for any $k\in P\cup C$ and $j\in J$ then, at delivery time $T_{j}$,
all power plants are at the upper or lower bound simultaneously. Thus,
$\left|\mu_{k,j}\right|<\infty$ for all $k\in P\cup C$ and $j\in J$.
Rewriting (\ref{eq:mt}) without focusing on one delivery period only
and taking any norm, we get
\begin{equation}
\left\Vert \mathbb{E}^{\tilde{\mathbb{P}}}\left[\Pi\right]\right\Vert \sim\left\Vert \lambda_{k}\hat{Q}_{1}V_{k}\right\Vert .\label{eq:mt-1}
\end{equation}
Since $\left|\lambda_{k}\right|<\infty$ and $\left\Vert \hat{Q}_{1}\right\Vert <\infty$,
it immediately follows from (\ref{eq:mt-1}) that if $\left\Vert \mathbb{E}^{\tilde{\mathbb{P}}}\left[\Pi\left(T_{j}\right)\right]\right\Vert \rightarrow\infty$
then $\left\Vert V_{k}\right\Vert \rightarrow\infty$. Because all
producers and consumers share the same price this must hold for all
of them simultaneously. Since $\hat{Q}_{1}$ is invertible, we can
write (\ref{eq:mt}) for all delivery periods together as
\begin{equation}
\begin{array}{rcl}
\left\Vert V_{k}\right\Vert  & \sim & \left\Vert \lambda_{k}^{-1}\hat{Q}_{1}\mathbb{E}^{\tilde{\mathbb{P}}}\left[\Pi\right]\right\Vert \end{array}.\label{eq:mt2}
\end{equation}
Because $\left|\lambda_{k}^{-1}\right|<\infty$ and $\left\Vert \hat{Q}_{1}^{-1}\right\Vert <\infty$,
it follows from (\ref{eq:mt2}) that if $\left\Vert V_{k}\right\Vert \rightarrow\infty$
then $\left\Vert \mathbb{E}^{\tilde{\mathbb{P}}}\left[\Pi\right]\right\Vert \rightarrow\infty$.
\end{proof}

\begin{lemma}\label{prop:bounded_prices}Expected prices $\mathbb{E}^{\tilde{\mathbb{P}}}\left[\Pi\right]$
in a NE are bounded.%
\footnote{Hence there exists some $M_{2}\in\mathbb{R}$ such that $\left\Vert \mathbb{E}^{\tilde{\mathbb{P}}}\left[\Pi\left(t_{i},T_{j}\right)\right]\right\Vert \leq M_{2}<\infty$
for all $j\in J$ and $i\in I_{j}$. We set $M_{2}<\Pi_{max}$.%
} \end{lemma}

\begin{proof}Assume that $\left\Vert \mathbb{E}^{\tilde{\mathbb{P}}}\left[\Pi\right]\right\Vert \rightarrow\infty$.
Then by Lemma \ref{prop:Unbound_vol_unbound_pr}, $\left[\sum_{k\in P\cup C}V_{k}\right]^{\top}\mathbb{E}^{\tilde{\mathbb{P}}}\left[\Pi\right]\rightarrow-\infty$.
Inserting that to (\ref{eq:utility_market_mx}), we get $\Psi_{M}\left(\mathbb{E}^{\tilde{\mathbb{P}}}\left[\Pi\right],V\right)=-\infty$,
which is clearly not a NE. Thus, $\left\Vert \mathbb{E}^{\tilde{\mathbb{P}}}\left[\Pi\right]\right\Vert <\infty$.
\end{proof}

\begin{proposition}\label{prop:NE-CErelation}A NE, if it exists,
is a CE and vice versa. \end{proposition}

\begin{proof}The equilibrium conditions for producers and consumers
are the same for both the NE and the CE. The remaining part is to
show that Point 3 from the CE implies Point 3 from the NE and conversely.
\begin{enumerate}
\item Assume the CE holds. Then (\ref{eq:vol}) can be inserted into (\ref{eq:market-1}
- \ref{eq:SO}). Then $\Psi_{M}\left(\mathbb{E}^{\tilde{\mathbb{P}}}\left[\Pi\right],V\right)=0$
for all $\mathbb{E}^{\mathbb{P}}\left[\Pi\right]$ and thus (\ref{eq:market})
is satisfied.
\item Assume the NE holds. Assume further that (\ref{eq:vol}) does not
hold for some $t_{i}$ and $T_{j}$, $i\in I_{j}$, $j\in J$. Then
according to (\ref{eq:SO}) the optimal $\left\Vert \mathbb{E}^{\tilde{\mathbb{P}}}\left[\Pi\left(t_{i},T_{j}\right)\right]\right\Vert \rightarrow\infty$.
This contradicts Lemma \ref{prop:bounded_prices} and hence such a
solution can not be a NE. Thus (\ref{eq:vol}) holds for all $t_{i}$
and $T_{j}$, $i\in I_{j}$, $j\in J$. 
\end{enumerate}
\end{proof}

\begin{corollary}\label{prop:bounded_volumes}The number of forward
contracts $V_{k}$ in a NE is bounded for all $k\in P\cup C$.%
\footnote{Hence there exists some $M_{3}\in\mathbb{R}$ such that $\left\Vert V_{k}\left(t_{i},T_{j}\right)\right\Vert \leq M_{3}<\infty$
for $k\in P\cup C$ and for all $j\in J$ and $i\in I_{j}$. We set
$M_{3}<V_{trade}$.%
} \end{corollary}

\begin{proof}Assume that there exists a player $k\in P\cup C$ such
that $\left\Vert V_{k}\right\Vert \rightarrow\infty$. By Proposition
\ref{prop:NE-CErelation} every NE is also a CE and thus (\ref{eq:vol})
holds. Thus, we can use Lemma \ref{prop:Unbound_vol_unbound_pr} and
conclude that $\left\Vert \mathbb{E}^{\tilde{\mathbb{P}}}\left[\Pi\right]\right\Vert \rightarrow\infty$.
This contradicts Lemma \ref{prop:bounded_prices}. Thus, $\left\Vert V_{k}\right\Vert <\infty$
for all $k\in P\cup C$. \end{proof}

Denote by $\bar{S}_{M}$ the set of prices for which power plants
are either all at the upper or all at the lower bound simultaneously
for at least one delivery period $T_{j}$ $j\in J$. By Lemma \ref{cor:cc1},
we know that the optimal price $\mathbb{E}^{\tilde{\mathbb{P}}}\left[\Pi\right]\notin\bar{S}_{M}$.

\begin{lemma}\label{prop:uniq_vol-1}Given an expected price vector
$\mathbb{E}^{\tilde{\mathbb{P}}}\left[\Pi\right]\in\mathbb{R}^{N}\backslash\bar{S}_{M}$,
the decision vectors $V_{p}$ for all $p\in P$ and the decision vectors
$V_{c}$ for all $c\in C$, are unique.\end{lemma}

\begin{proof}The objective function of each producer $p\in P$ can
be written as 
\[
\Psi_{p}\left(v_{p},\mathbb{E}^{\tilde{\mathbb{P}}}\left[\Pi\right]\right)=-\mathbb{E}^{\tilde{\mathbb{P}}}\left[\pi_{p}\right]^{\top}v_{p}-\frac{1}{2}\lambda_{p}v_{p}^{\top}Q_{p}v_{p}.
\]
Define $v_{p}':=V_{p}$ and $v''_{p}:=\left[F_{p}^{\top},O_{p}^{\top},W_{p}^{\top}\right]^{\top}$.
Then 
\[
\mathcal{D}_{v_{p}}\Psi_{p}\left(v_{p},\mathbb{E}^{\tilde{\mathbb{P}}}\left[\Pi\right]\right)=\left[\mathcal{D}_{v_{p'}}\Psi_{p}\left(v_{p},\mathbb{E}^{\tilde{\mathbb{P}}}\left[\Pi\right]\right)^{\top},\mathcal{D}_{v_{p''}}\Psi_{p}\left(v_{p},\mathbb{E}^{\tilde{\mathbb{P}}}\left[\Pi\right]\right)^{\top}\right]^{\top}.
\]
Due to the strict concavity of the expected utility functions in $v_{p}'$,
for any $\hat{v}_{p}:=\left[\hat{v}_{p}'^{\top},\hat{v}_{p}''^{\top}\right]^{\top}$
and $\tilde{v}_{p}:=\left[\tilde{v}_{p}'^{\top},\tilde{v}_{p}''^{\top}\right]^{\top}$
with $\hat{v}_{p}'\neq\tilde{v}_{p}'$ the following strict inequality
holds 
\begin{equation}
\left(\hat{v}_{p}-\tilde{v}_{p}\right)^{\top}\mathcal{D}_{\tilde{v}_{p}}\Psi_{p}\left(\tilde{v}_{p},\mathbb{E}^{\tilde{\mathbb{P}}}\left[\Pi\right]\right)+\left(\tilde{v}_{p}-\hat{v}_{p}\right)^{\top}\mathcal{D}_{\hat{v}_{p}}\Psi_{p}\left(\hat{v}_{p},\mathbb{E}^{\tilde{\mathbb{P}}}\left[\Pi\right]\right)>0.\label{eq:ineqQ-1}
\end{equation}

We will continue with a proof by contradiction. Assume that there
exist $\hat{v}_{p}'\neq\tilde{v}_{p}'$ that are both optimal solutions
for player $p\in P$ given the electricity price $\mathbb{E}^{\tilde{\mathbb{P}}}\left[\Pi\right]$.
Then both must satisfy the KKT conditions, i.e. 
\begin{equation}
\begin{array}{rcl}
\mathcal{D}_{\tilde{v}_{p}}\Psi_{p}\left(\tilde{v}_{p},\mathbb{E}^{\tilde{\mathbb{P}}}\left[\Pi\right]\right)-B_{p}^{\top}\tilde{\eta}_{p}-A_{p}^{\top}\tilde{\mu}_{p} & = & 0\\
\\
\tilde{\eta}_{p}^{\top}\left(B_{p}\tilde{v}_{p}-b_{p}\right) & = & 0\\
\\
B_{p}\tilde{v}_{p}-b_{p} & \leq & 0\\
\\
A_{p}\tilde{v}_{p}-a_{p} & = & 0\\
\\
\tilde{\eta}_{p} & \geq & 0
\end{array}\label{eq:KKT-dir1-1}
\end{equation}
and
\begin{equation}
\begin{array}{rcl}
\mathcal{D}_{\hat{v}_{p}}\Psi_{p}\left(\hat{v}_{p},\mathbb{E}^{\tilde{\mathbb{P}}}\left[\Pi\right]\right)-B_{p}^{\top}\hat{\eta}_{p}-A_{p}^{\top}\hat{\mu}_{p} & = & 0\\
\\
\hat{\eta}_{p}^{\top}\left(B_{p}\hat{v}_{p}-b_{p}\right) & = & 0\\
\\
B_{p}\hat{v}_{p}-b_{p} & \leq & 0\\
\\
A_{p}\hat{v}_{p}-a_{p} & = & 0\\
\\
\hat{\eta}_{p} & \geq & 0.
\end{array}\label{eq:KKT-dir2-1}
\end{equation}
Multiplying the first equation of (\ref{eq:KKT-dir1-1}) and (\ref{eq:KKT-dir2-1})
by $\left(\hat{v}_{p}-\tilde{v}_{p}\right)^{\top}$ and $\left(\tilde{v}_{p}-\hat{v}_{p}\right)^{\top}$,
respectively and summing them up, the following strict inequality
\begin{equation}
\begin{array}{rcl}
0 & = & \left(\hat{v}_{p}-\tilde{v}_{p}\right)^{\top}\mathcal{D}_{\tilde{v}_{p}}\Psi_{p}\left(\tilde{v}_{p},\mathbb{E}^{\tilde{\mathbb{P}}}\left[\Pi\right]\right)+\left(\tilde{v}_{p}-\hat{v}_{p}\right)^{\top}\mathcal{D}_{\hat{v}_{p}}\Psi_{p}\left(\hat{v}_{p},\mathbb{E}^{\tilde{\mathbb{P}}}\left[\Pi\right]\right)\\
\\
 &  & -\left(\hat{v}_{p}-\tilde{v}_{p}\right)^{\top}B_{p}^{\top}\tilde{\eta}_{p}-\left(\tilde{v}_{p}-\hat{v}_{p}\right)^{\top}B_{p}^{\top}\hat{\eta}_{p}\\
\\
 &  & -\left(\hat{v}_{p}-\tilde{v}_{p}\right)^{\top}A_{p}^{\top}\tilde{\mu}_{p}^{\top}-\left(\tilde{v}_{p}-\hat{v}_{p}\right)^{\top}A_{p}^{\top}\hat{\mu}_{p}^{\top}\\
\\
 & > & -\left(\hat{v}_{p}-\tilde{v}_{p}\right)^{\top}B_{p}^{\top}\tilde{\eta}_{p}-\left(\tilde{v}_{p}-\hat{v}_{p}\right)^{\top}B_{p}^{\top}\hat{\eta}_{p}\\
\\
 &  & -\left(\hat{v}_{p}-\tilde{v}_{p}\right)^{\top}A_{p}^{\top}\tilde{\mu}_{p}-\left(\tilde{v}_{p}-\hat{v}_{p}\right)^{\top}A_{p}^{\top}\hat{\mu}_{p}
\end{array}\label{eq:m1-2-1-1}
\end{equation}
is obtained by (\ref{eq:ineqQ-1}).

Rewriting 
\[
\begin{array}{rcl}
\left(\hat{v}_{p}-\tilde{v}_{p}\right)^{\top}B_{p}^{\top}\tilde{\eta}_{p} & = & \left(B_{p}\hat{v}_{p}-B_{p}\tilde{v}_{p}+b_{p}-b_{p}\right)^{\top}\tilde{\eta}_{p}\\
\\
 & = & \left(B_{p}\hat{v}_{p}-b_{p}\right)^{\top}\tilde{\eta}_{p}-\left(B_{p}\tilde{v}_{p}-b_{p}\right)^{\top}\tilde{\eta}_{p}\\
\\
 & = & \left(B_{p}\hat{v}_{p}-b_{p}\right)^{\top}\tilde{\eta}_{p}
\end{array}
\]
and noting $\tilde{\eta}_{p}\geq0$ and $B_{p}\hat{v}_{p}-b_{p}\leq0$,
we can conclude that 
\begin{equation}
\left(\hat{v}_{p}-\tilde{v}_{p}\right)^{\top}B_{p}^{\top}\tilde{\eta}_{p}\leq0.\label{eq:mm1}
\end{equation}
Due to the symmetry also
\begin{equation}
\left(\tilde{v}_{p}-\hat{v}_{p}\right)^{\top}B_{p}^{\top}\hat{\eta}_{p}\leq0.\label{eq:mm2}
\end{equation}
Rewriting
\begin{equation}
\begin{array}{rcl}
\left(\hat{v}_{p}-\tilde{v}_{p}\right)^{\top}A_{p}^{\top}\tilde{\mu}_{p} & = & \left(A_{p}\hat{v}_{p}-A_{p}\tilde{v}_{p}+a_{p}-a_{p}\right)^{\top}\tilde{\mu}_{p}\\
\\
 & = & \left(A_{p}\hat{v}_{p}-a_{p}\right)^{\top}\tilde{\mu}_{p}-\left(A_{p}\tilde{v}_{p}-a_{p}\right)^{\top}\tilde{\mu}_{p}\\
\\
 & = & 0
\end{array}\label{eq:mm3}
\end{equation}
and due to the symmetry also 
\begin{equation}
\left(\tilde{v}_{p}-\hat{v}_{p}\right)^{\top}A_{p}^{\top}\hat{\mu}_{p}=0.\label{eq:mm4}
\end{equation}
Inserting (\ref{eq:mm1}), (\ref{eq:mm2}), (\ref{eq:mm3}), and (\ref{eq:mm4})
back to (\ref{eq:m1-2-1-1}) gives a contradiction. Hence $\hat{v}_{p}'=\tilde{v}_{p}'$.
The proof for consumers $c\in C$ is similar. \end{proof}

Before we continue with a further analysis of our problem, let us
introduce some definitions that are useful for an analysis of piecewise
differential functions. 

\begin{definition}(\cite{ralph1997sensitivity}) A continuous function
$\tilde{\mathcal{Z}}:S_{1}\rightarrow\mathbb{R}^{N}$ defined on an
open set $S_{1}\subseteq\mathbb{R}^{N}$ is $\mathbb{PC}^{r}$ if
for every $x\in S_{1}$ there exists a finite family of $\mathbb{C}^{r}$-functions
$\tilde{\mathcal{Z}}^{i}:S_{2}\rightarrow\mathbb{R}^{N}$, where $S_{2}\subseteq S_{1}$
is an open neighborhood of $x$ and $i\in\mathcal{I}$, such that
$\tilde{\mathcal{Z}}\left(z\right)\in\left\{ \tilde{\mathcal{Z}}^{i}\left(z\right):i\in\mathcal{I}\right\} $
for every $z\in S_{2}$. The $\mathbb{C}^{r}$-functions $\tilde{\mathcal{Z}}^{i}$,
$i\in\mathcal{I}$, are called selection functions of $\tilde{\mathcal{Z}}$
at $x$. 

A selection function $\tilde{\mathcal{Z}}^{i}$ of a $\mathbb{PC}^{r}$-
function $\tilde{\mathcal{Z}}$ at $x$ is essentially active if
\begin{equation}
x\in\text{cl}\,\text{int}\left\{ z\in S_{2}:\tilde{\mathcal{Z}}^{i}\left(z\right)=\tilde{\mathcal{Z}}\left(z\right)\right\} .
\end{equation}

If $\left\{ \tilde{\mathcal{Z}}^{i}:i\in\mathcal{I}\right\} $ is
a family of selection functions for $\tilde{\mathcal{Z}}$ at $x$,
then the set of indices $i$ of essentially active selection functions
$\tilde{\mathcal{Z}}^{i}$ at $x$ is denoted $\mathcal{I}^{e}\left(x\right)$.
\end{definition}

\begin{definition}Let $\tilde{\mathcal{Z}}:S_{1}\rightarrow\mathbb{R}^{N}$
be a continuous $\mathbb{PC}^{1}$ function defined on an open set
$S_{1}\subseteq\mathbb{R}^{N}$. Clarke's generalized Jacobian (\cite{clarke1990optimization})
of $\tilde{\mathcal{Z}}$ at $x$ is defined as
\begin{equation}
\mathcal{D^{C}}\tilde{\mathcal{Z}}\left(x\right)=\text{conv}\left\{ \mathcal{D}\tilde{\mathcal{Z}}^{i}\left(x\right):i\in\mathcal{I}^{e}\left(x\right)\right\} \label{eq:clarke}
\end{equation}
where $\text{conv}$ denotes a convex hull. \end{definition}

\begin{proposition}\label{prop:PD}$\mathcal{D}\tilde{\mathcal{Z}}^{i}\left(x\right)\prec0$
for all $i\in\mathcal{I}^{e}\left(x\right)$ and all $x\in S_{1}$
if and only if $\mathcal{J}\left(x\right)\prec0$ for all $\mathcal{J}\left(x\right)\in\mathcal{D^{C}}\tilde{\mathcal{Z}}\left(x\right)$.
\end{proposition}

\begin{proof}This follows directly from (\ref{eq:clarke}). For any
$\mathcal{J}\left(x\right)\in\mathcal{D^{C}}\tilde{\mathcal{Z}}\left(x\right)$
and $x\in S_{1}$, $\mathcal{J}\left(x\right)$ can be written as
\begin{equation}
\mathcal{J}\left(x\right)=\sum_{i\in\mathcal{I}^{e}\left(x\right)}\alpha_{i}\mathcal{D}\tilde{\mathcal{Z}}^{i}\left(x\right)\label{eq:J(x)}
\end{equation}
for some $\alpha_{i}\geq0$ such that $\sum_{i\in\mathcal{I}^{e}\left(x\right)}\alpha_{i}=1$.
The result then follows trivially. \end{proof}

\begin{definition}A function $\tilde{\mathcal{Z}}:S_{1}\rightarrow\mathbb{R}^{N}$
defined on an open set $S_{1}\subseteq\mathbb{R}^{N}$ is strictly
decreasing on $S_{1}$ if
\begin{equation}
\left(\tilde{\mathcal{Z}}\left(x\right)-\tilde{\mathcal{Z}}\left(y\right)\right)^{\top}\left(x-y\right)<0
\end{equation}
for all $x,y\in S_{1}$ such that $x\neq y$.\end{definition}

\begin{proposition}\label{prop:monotone}Let $\tilde{\mathcal{Z}}:S_{1}\rightarrow\mathbb{R}^{N}$
be a $\mathbb{PC}^{1}$ function defined on an open set $S_{1}\subseteq\mathbb{R}^{N}$.
If $\mathcal{D}\tilde{\mathcal{Z}}^{i}\left(x\right)\prec0$ for all
$i\in\mathcal{I}^{e}\left(x\right)$ and $x\in S_{1}$, then $\tilde{\mathcal{Z}}$
is strictly decreasing on $S_{1}$. \end{proposition}

\begin{proof}Assume that $\mathcal{D}\tilde{\mathcal{Z}}^{i}\left(x\right)\prec0$
for all $i\in\mathcal{I}^{e}\left(x\right)$. Then by Proposition
\ref{prop:PD} also $\mathcal{J}\left(x\right)\prec0$ for all $\mathcal{J}\left(x\right)\in\mathcal{D^{C}}\tilde{\mathcal{Z}}\left(x\right)$.
By the extended mean-value theorem (see \cite{jeyakumar1998approximate})
\begin{equation}
\tilde{\mathcal{Z}}\left(x\right)-\tilde{\mathcal{Z}}\left(y\right)=\mathcal{J}\left(y+\delta\left(x-y\right)\right)\left(x-y\right)\label{eq:mn1-1}
\end{equation}
for some $\mathcal{J}\left(y+\delta\left(x-y\right)\right)\in\mathcal{D^{C}}\tilde{\mathcal{Z}}\left(y+\delta\left(x-y\right)\right)$
and $\delta\in\left(0,1\right)$. By multiplying (\ref{eq:mn1-1})
from the left by $\left(x-y\right)^{\top}$ and using the assumption
of the negative definiteness of $\mathcal{J}\left(x\right)$, we conclude
\begin{equation}
\left(\tilde{\mathcal{Z}}\left(x\right)-\tilde{\mathcal{Z}}\left(y\right)\right)^{\top}\left(x-y\right)<0.
\end{equation}
\end{proof}

\begin{lemma}\label{prop:trans1-1}Let $k\in P\cup C$. There exists
a $\mathbb{PC}^{\infty}$ mapping $\tilde{\mathcal{Z}}_{k}:\mathbb{R}^{N}\backslash\bar{S}_{M}\rightarrow\mathbb{R}^{N}$
that maps the electricity price vector $\mathbb{E}^{\tilde{\mathbb{P}}}\left[\Pi\right]\in\mathbb{R}^{N}\backslash\bar{S}_{M}$
to a volume vector $V_{k}\in\mathbb{R}^{N}$. \end{lemma}

\begin{proof}Since the proof for producers and consumers is almost
the same we will explicitly write it only for producers. Set $p=k$
for some $p\in P$. The optimization problem for a producer $p$ can
be written as
\[
\begin{array}{rl}
\underset{v_{p}\in S_{p}}{\text{max}} & \Psi_{p}\left(v_{p},\mathbb{E}^{\tilde{\mathbb{P}}}\left[\Pi\right]\right).\end{array}
\]
For convenience we define $\tilde{\Psi}_{p}:\mathbb{R}^{N}\times\mathbb{R}^{N}\rightarrow\mathbb{R}$
as $\tilde{\Psi}_{p}\left(V_{p},\mathbb{E}^{\tilde{\mathbb{P}}}\left[\Pi\right]\right):=\underset{F_{p},O_{p},W_{p}}{\text{max}}\Psi_{p}\left(v_{p},\mathbb{E}^{\tilde{\mathbb{P}}}\left[\Pi\right]\right)$
subject to (\ref{eq:pb0}), (\ref{eq:pb1}), (\ref{eq:pb2-1}), (\ref{eq:pb3-1}),
(\ref{eq:pb4}), (\ref{eq:pb3}), and (\ref{eq:pbb4}). Using Lemma
\ref{prop:uniq_vol-1} we know that volumes $V_{p}$ are unique for
a given price $\mathbb{E}^{\tilde{\mathbb{P}}}\left[\Pi\right]\in\mathbb{R}^{N}\backslash\bar{S}_{M}$.
Thus there exists a mapping $\tilde{\mathcal{Z}}_{p}:\mathbb{R}^{N}\backslash\bar{S}_{M}\rightarrow\mathbb{R}^{N}$
such that $V_{p}=\tilde{\mathcal{Z}}_{p}\left(\mathbb{E}^{\tilde{\mathbb{P}}}\left[\Pi\right]\right)$.
Results from the parametric quadratic programming show that $\tilde{\mathcal{Z}}_{p}\left(\mathbb{E}^{\tilde{\mathbb{P}}}\left[\Pi\right]\right)$
is a continuous and piecewise affine function (and thus $\mathbb{PC}^{\infty}$),
see \cite{boot1963onsensitivity} for a strictly convex quadratic
objective function and \cite{berkelaar1996sensitivity} for a convex
quadratic objective function. \end{proof}

\begin{theorem}\label{thm:uniq} Define a $\mathbb{PC}^{\infty}$
mapping $\tilde{\mathcal{Z}}\left(\mathbb{E}^{\tilde{\mathbb{P}}}\left[\Pi\right]\right):=\sum_{p\in P}\tilde{\mathcal{Z}}_{p}\left(\mathbb{E}^{\tilde{\mathbb{P}}}\left[\Pi\right]\right)+\sum_{c\in C}\tilde{\mathcal{Z}}_{c}\left(\mathbb{E}^{\tilde{\mathbb{P}}}\left[\Pi\right]\right)$.
If $\mathcal{D}\tilde{\mathcal{Z}}^{i}\left(x\right)\prec0$ for all
$i\in\mathcal{I}^{e}\left(x\right)$ and $x\in\mathbb{R}^{N}\backslash\bar{S}_{M}$,
then electricity prices $\mathbb{E}^{\tilde{\mathbb{P}}}\left[\Pi\right]$
in the NE are unique. \end{theorem}

\begin{proof}The market condition (\ref{eq:vol}) requires that
\[
0=\sum_{c\in C}V_{c}+\sum_{p\in P}V_{p}.
\]
Using Lemma \ref{prop:trans1-1} this can be written as
\begin{equation}
\begin{array}{rcl}
0 & = & \sum_{p\in P}V_{p}+\sum_{c\in C}V_{c}\\
\\
 & = & \sum_{p\in P}\tilde{\mathcal{Z}}_{p}\left(\mathbb{E}^{\tilde{\mathbb{P}}}\left[\Pi\right]\right)+\sum_{c\in C}\tilde{\mathcal{Z}}_{c}\left(\mathbb{E}^{\tilde{\mathbb{P}}}\left[\Pi\right]\right)\\
\\
 & = & \tilde{\mathcal{Z}}\left(\mathbb{E}^{\tilde{\mathbb{P}}}\left[\Pi\right]\right).
\end{array}\label{eq:mk-1}
\end{equation}
If $\mathcal{D}\tilde{\mathcal{Z}}^{i}\left(x\right)\prec0$ for all
$i\in\mathcal{I}^{e}\left(x\right)$ and $x\in\mathbb{R}^{N}\backslash\bar{S}_{M}$,
then by Proposition \ref{prop:monotone} $\tilde{\mathcal{Z}}$ is
strictly decreasing on $\mathbb{R}^{N}\backslash\bar{S}_{M}$. Thus,
the mapping $\tilde{\mathcal{Z}}$ must be bijective. By Corollary
\ref{cor:exist}, we know that the solution to the problem $\tilde{\mathcal{Z}}\left(\mathbb{E}^{\tilde{\mathbb{P}}}\left[\Pi\right]\right)=0$
exists and from the bijectivity of $\tilde{\mathcal{Z}}$, we conclude
that it must be unique. \end{proof}

The remaining problem of this section is to show that $\mathcal{D}\tilde{\mathcal{Z}}^{i}\left(\mathbb{E}^{\tilde{\mathbb{P}}}\left[\Pi\right]\right)\prec0$
for all $i\in\mathcal{I}^{e}\left(\mathbb{E}^{\tilde{\mathbb{P}}}\left[\Pi\right]\right)$
and $\mathbb{E}^{\tilde{\mathbb{P}}}\left[\Pi\right]\in\mathbb{R}^{N}\backslash\bar{S}_{M}$
indeed holds. From the proof of Lemma \ref{prop:trans1-1}, we know
that $\tilde{\mathcal{Z}}^{i}\left(\mathbb{E}^{\tilde{\mathbb{P}}}\left[\Pi\right]\right)$,
$i\in\mathcal{I}$ are all affine functions and thus $\mathcal{D}\tilde{\mathcal{Z}}^{i}\left(\mathbb{E}^{\tilde{\mathbb{P}}}\left[\Pi\right]\right)$
are all constants. Therefore, it is enough to show that $\mathcal{D}\tilde{\mathcal{Z}}^{i}\left(\mathbb{E}^{\tilde{\mathbb{P}}}\left[\Pi\right]\right)\prec0$
for all $i\in\mathcal{I}^{e}\left(\mathbb{E}^{\tilde{\mathbb{P}}}\left[\Pi\right]\right)$
and $\mathbb{E}^{\tilde{\mathbb{P}}}\left[\Pi\right]\in\left\{ x:x\in\mathbb{R}^{N}\backslash\bar{S}_{M}\wedge\left|\mathcal{I}^{e}\left(x\right)\right|=1\right\} $.
For such prices $\mathbb{E}^{\tilde{\mathbb{P}}}\left[\Pi\right]$,
$\mathcal{D}\tilde{\mathcal{Z}}_{k}\left(\mathbb{E}^{\tilde{\mathbb{P}}}\left[\Pi\right]\right)$
exists for all players $k\in P\cup C$. In the rest of this section
we simplify the notation and omit writing the dependence on $\mathbb{E}^{\tilde{\mathbb{P}}}\left[\Pi\right]\in\left\{ x:x\in\mathbb{R}^{N}\backslash\bar{S}_{M}\wedge\left|\mathcal{I}^{e}\left(x\right)\right|=1\right\} $
explicitly. All the statements hold for any such $\mathbb{E}^{\tilde{\mathbb{P}}}\left[\Pi\right]$.

For the further argumentation, we introduce $R\left(\cdot\right)$
and $N\left(\cdot\right)$ that denote a range and a null space, respectively.

\begin{lemma}\label{prop:np}Let $x\in\mathbb{R}^{N}\backslash\left\{ 0\right\} $
and $c\in C$ be any of the consumers. Then $x^{\top}\mathcal{D}\tilde{\mathcal{Z}}_{c}x\leq0$.
Moreover, $x^{\top}\mathcal{D}\tilde{\mathcal{Z}}_{c}x=0$ if and
only if $x\in R\left(\hat{A}_{1}^{\top}\right)$. \end{lemma}

\begin{proof}Since the optimal power trading strategy $V_{c}$ can
not be at the boundary of the feasible region $S_{c}$ due to Lemma
\ref{prop:bounded_volumes}, the necessary and due to convexity and
the Slater condition in Assumption \ref{ass: as1} also sufficient
KKT conditions for $V_{c}$ to be a global minimizer given the forward
electricity price reads
\begin{equation}
\begin{array}{rcl}
-\mathbb{E}^{\tilde{\mathbb{P}}}\left[\Pi\right]-\lambda_{c}\hat{Q}_{1}V_{c}-\hat{A}_{1}^{\top}\mu_{c} & = & 0\\
\\
\hat{A}_{1}V_{c} & = & a_{c}.
\end{array}\label{eq:KKTpr}
\end{equation}
One can then solve for $V_{c}$ from the first equation of (\ref{eq:KKTpr})
by multiplying both sides from the left by $\left(\lambda_{c}\hat{Q}_{1}\right)^{-1}$
and obtain 
\begin{equation}
V_{c}=-\frac{1}{\lambda_{c}}\hat{Q}_{1}^{-1}\mathbb{E}^{\tilde{\mathbb{P}}}\left[\Pi\right]-\frac{1}{\lambda_{c}}\hat{Q}_{1}^{-1}\hat{A}_{1}^{\top}\mu_{c}.\label{eq:ml}
\end{equation}
By further multiplying both sides of (\ref{eq:ml}) from the left
by $\hat{A}_{1}$, we get 
\begin{equation}
\hat{A}_{1}V_{c}=-\frac{1}{\lambda_{c}}\hat{A}_{1}\hat{Q}_{1}^{-1}\mathbb{E}^{\tilde{\mathbb{P}}}\left[\Pi\right]-\frac{1}{\lambda_{c}}\hat{A}_{1}\hat{Q}_{1}^{-1}\hat{A}_{1}^{\top}\mu_{c}.
\end{equation}
Since $\hat{A}_{1}$ has full row rank $\left(\hat{A}_{1}\hat{Q}_{1}^{-1}\hat{A}_{1}^{\top}\right)^{-1}$
exists. By using the second equation of (\ref{eq:KKTpr}), we get
\begin{equation}
\mu_{c}=-\left(\hat{A}_{1}\hat{Q}_{1}^{-1}\hat{A}_{1}^{\top}\right)^{-1}\left(\lambda_{c}a_{c}+\hat{A}_{1}\hat{Q}_{1}^{-1}\mathbb{E}^{\tilde{\mathbb{P}}}\left[\Pi\right]\right).
\end{equation}
Inserting that back into (\ref{eq:ml}), we can calculate 
\begin{equation}
\frac{\partial V_{c}}{\partial\mathbb{E}^{\tilde{\mathbb{P}}}\left[\Pi\right]}=-\frac{1}{\lambda_{c}}\hat{Q}_{1}^{-1}+\frac{1}{\lambda_{c}}\hat{Q}_{1}^{-1}\hat{A}_{1}^{\top}\left(\hat{A}_{1}\hat{Q}_{1}^{-1}\hat{A}_{1}^{\top}\right)^{-1}\hat{A}_{1}\hat{Q}_{1}^{-1}.\label{eq:ml1-1}
\end{equation}
Since $\hat{Q}_{1}$ (and consequently also $\hat{Q}_{1}^{-1}$) is
positive definite and symmetric, we can use the Cholesky decomposition
to define a matrix $\tilde{Q_{1}}\in\mathbb{R}^{N\times N}$ such
that $\hat{Q}_{1}^{-1}=\tilde{Q}_{1}^{-1}\tilde{Q_{1}}^{-\top}$.
Since $\tilde{Q_{1}}$ is invertible, we can rewrite (\ref{eq:ml1-1})
as 
\begin{equation}
\tilde{P}:=-\lambda_{c}\tilde{Q}_{1}\frac{\partial V_{c}}{\partial\mathbb{E}^{\tilde{\mathbb{P}}}\left[\Pi\right]}\tilde{Q_{1}}^{\top}=I-\tilde{Q_{1}}^{-\top}\hat{A}_{1}^{\top}\left(\hat{A}_{1}\hat{Q}_{1}^{-1}\hat{A}_{1}^{\top}\right)^{-1}\hat{A}_{1}\tilde{Q_{1}}^{-1}.\label{eq:mtmt}
\end{equation}
It is trivial to check that $\tilde{P}=\tilde{P}^{2}$ (i.e. $\tilde{P}$
is idempotent) and symmetric. Thus, $\tilde{P}$ must be a projection
matrix. It is known that every projection matrix is positive semidefinite
and thus $\frac{\partial V_{c}}{\partial\mathbb{E}^{\tilde{\mathbb{P}}}\left[\Pi\right]}\preceq0$.
Moreover, for any $x\in\mathbb{R}^{N}\backslash\left\{ 0\right\} $,
$x^{\top}\tilde{P}x=0$ if and only if $x\in R\left(\tilde{Q_{1}}^{-\top}\hat{A}_{1}^{\top}\right)$.
Therefore, it follows from (\ref{eq:mtmt}) that $x^{\top}\frac{\partial V_{c}}{\partial\mathbb{E}^{\tilde{\mathbb{P}}}\left[\Pi\right]}x=0$
if and only if $x\in R\left(\hat{A}_{1}^{\top}\right)$.\end{proof}

\begin{lemma}Assume that $x^{\top}\mathcal{D}\tilde{\mathcal{Z}}_{p}x\leq0$
for all $p\in P$ and all $x\in\mathbb{R}^{N}\backslash\left\{ 0\right\} $.
Then $\mathcal{D}\tilde{\mathcal{Z}}\prec0$ if and only if $\sum_{p\in P}\hat{A}_{1}\mathcal{D}\tilde{\mathcal{Z}}_{p}\hat{A}_{1}^{\top}$
has a full rank. \end{lemma}

\begin{proof}We can write 
\[
x^{\top}\mathcal{D}\tilde{\mathcal{Z}}x=\sum_{k\in P\cup C}x^{\top}\mathcal{D}\tilde{\mathcal{Z}}_{p}x.
\]
By Lemma \ref{prop:np} and the assumption that $x^{\top}\mathcal{D}\tilde{\mathcal{Z}}_{p}x\leq0$
for all $p\in P$ and all $x\in\mathbb{R}^{N}\backslash\left\{ 0\right\} $,
we can see that $x^{\top}\mathcal{D}\tilde{\mathcal{Z}}_{k}x\leq0$
for all $k\in P\cup C$. Thus, $x^{\top}\mathcal{D}\tilde{\mathcal{Z}}x<0$
if and only if for each $x\in\mathbb{R}^{N}\backslash\left\{ 0\right\} $
there exists at least one player $k'\in P\cup C$ such that $x^{\top}\mathcal{D}\tilde{\mathcal{Z}}_{k'}x<0$.
If $x\notin R\left(\hat{A}_{1}^{\top}\right)$ then $x^{\top}\mathcal{D}\tilde{\mathcal{Z}}_{k}x<0$
for all $k\in C$. On the other hand, if $x\in R\left(\hat{A}_{1}^{\top}\right)$
then $x^{\top}\mathcal{D}\tilde{\mathcal{Z}}_{k}x=0$ for all $k\in C$.
Thus, for all $x\in R\left(\hat{A}_{1}^{\top}\right)$, $\mathcal{D}\tilde{\mathcal{Z}}\prec0$
if and only if $\sum_{p\in P}\hat{A}_{1}\mathcal{D}\tilde{\mathcal{Z}}_{p}\hat{A}_{1}^{\top}$
has a full rank.\end{proof}

\begin{lemma}\label{prop:Inv}Let $A\in\mathbb{R}^{n\times n}$,
$B\in\mathbb{R}^{n\times m}$, and $C\in\mathbb{R}^{m\times m}$ be
real matrices such that $A^{-1}$ and $\left(C-BC^{-1}B^{\top}\right)$
exist. Then
\[
\left[\begin{array}{cc}
A^{\top} & B^{\top}\end{array}\right]\left[\begin{array}{cc}
A & B\\
B^{\top} & C
\end{array}\right]^{-1}=\left[\begin{array}{cc}
I & 0\end{array}\right]
\]
and consequently 
\[
\left[\begin{array}{cc}
A^{\top} & B^{\top}\end{array}\right]\left[\begin{array}{cc}
A & B\\
B^{\top} & C
\end{array}\right]^{-1}\left[\begin{array}{c}
A\\
B
\end{array}\right]=A.
\]
\end{lemma} 

\begin{proof}Since $A^{-1}$ and $\left(C-BC^{-1}B^{\top}\right)$
exist, we can use the block matrix inverse formula to calculate
\[
\left[\begin{array}{cc}
A & B\\
B^{\top} & C
\end{array}\right]^{-1}=\left[\begin{array}{cc}
A^{-1}+A^{-1}B\left(C-B^{\top}A^{-1}B\right)^{-1}B^{\top}A^{-1} & -A^{-1}B\left(C-B^{\top}A^{-1}B\right)^{-1}\\
-\left(C-B^{\top}A^{-1}B\right)^{-1}B^{\top}A^{-1} & \left(C-B^{\top}A^{-1}B\right)^{-1}
\end{array}\right].
\]
The result then follows trivially.\end{proof}

\begin{theorem}For all $x\in\mathbb{R}^{N}$ and all $p\in P$, $x^{\top}\mathcal{D}\tilde{\mathcal{Z}}_{p}x\leq0$
holds. Moreover, if for each delivery period $j\in J$ there exist
at least one power plant that has a strictly feasible optimal production,
then $\sum_{p\in P}\hat{A}_{1}\mathcal{D}\tilde{\mathcal{Z}}_{p}\hat{A}_{1}^{\top}$
has a full rank. \end{theorem}

\begin{proof}We can assume w.l.o.g that there exists a producer $p\in P$
who owns a set of power plants among which at least one has a strictly
feasible optimal production in each delivery periods $j\in J$. In
this case, showing that $\sum_{p\in P}\hat{A}_{1}\mathcal{D}\tilde{\mathcal{Z}}_{p}\hat{A}_{1}^{\top}$
has a full rank simplifies to showing that $\hat{A}_{1}\mathcal{D}\tilde{\mathcal{Z}}_{p}\hat{A}_{1}^{\top}$
has a full rank. 

Assume that the set of active inequality constraints $B_{p}$ is known.
Constraints (\ref{eq:pb3}) and (\ref{eq:pbb4}) are never active
due to Lemma \ref{prop:Optimal-fuel-trading}. Similarly, Constraints
(\ref{eq:pb2}) are never active due to Lemma \ref{prop:bounded_volumes}.
The KKT conditions can thus be written as
\begin{equation}
\begin{array}{rcl}
-\mathbb{E}^{\tilde{\mathbb{P}}}\left[\pi\right]-\lambda_{p}\hat{Q}v_{p}'-\hat{A}_{12}^{\top}\mu'_{p} & = & 0\\
\\
-\hat{A}_{p}^{\top}\mu'_{p}-B_{p}^{\top}\eta'_{p} & = & 0\\
\\
\hat{A}_{12}v_{p}'+\hat{A}_{p}v''_{p} & = & 0\\
\\
B_{p}v''_{p} & = & b
\end{array}\label{eq:KKTpr-1}
\end{equation}
where $v_{p}':=\left[V_{p}^{\top},F_{p}^{\top},O_{p}^{\top}\right]^{\top}$
and $v_{p}'':=W_{p}$. The internal structure of the matrices is the
following 
\begin{equation}
\hat{Q}=:\left[\begin{array}{cc}
\hat{Q}_{1} & \hat{Q}_{2}\\
\hat{Q}_{2}^{\top} & \hat{Q}_{3}
\end{array}\right],\quad\tilde{A}_{12}:=\left[\begin{array}{cc}
\hat{A}_{1} & 0\\
0 & \hat{A}_{2}
\end{array}\right],\quad\hat{A}_{p}:=\left[\begin{array}{c}
\hat{A}_{3,p}\\
\hat{A}_{4,p}
\end{array}\right].\label{eq:insStr}
\end{equation}
Since $\hat{Q}\succ0$, we can express $v_{p}'$ from the first equation
of (\ref{eq:KKTpr-1}) as
\begin{equation}
v_{p}'=-\frac{1}{\lambda_{p}}\hat{Q}^{-1}\mathbb{E}^{\tilde{\mathbb{P}}}\left[\pi\right]-\frac{1}{\lambda_{p}}\hat{Q}^{-1}\hat{A}_{12}^{\top}\mu'_{p}\label{eq:bl1}
\end{equation}
and further $\hat{A}_{12}v_{p}'$ as
\begin{equation}
\hat{A}_{12}v_{p}'=-\frac{1}{\lambda_{p}}\hat{A}_{12}\hat{Q}^{-1}\mathbb{E}^{\tilde{\mathbb{P}}}\left[\pi\right]-\frac{1}{\lambda_{p}}\hat{A}_{12}\hat{Q}^{-1}\hat{A}_{12}^{\top}\mu'_{p}.
\end{equation}
Using the third equation of (\ref{eq:KKTpr-1}) this reads 
\begin{equation}
\hat{A}_{p}v''_{p}=\frac{1}{\lambda_{p}}\hat{A}_{12}\hat{Q}^{-1}\mathbb{E}^{\tilde{\mathbb{P}}}\left[\pi\right]+\frac{1}{\lambda_{p}}\hat{A}_{12}\hat{Q}^{-1}\hat{A}_{12}^{\top}\mu'_{p}.
\end{equation}
Since $\hat{A}_{1}$ and $\hat{A}_{2}$ both have a full rank, also
$\hat{A}_{12}$ has a full rank and thus $\left(\hat{A}_{12}\hat{Q}^{-1}\hat{A}_{12}^{\top}\right)^{-1}$
exists. Then 
\begin{equation}
\mu'_{p}=\left(\hat{A}_{12}\hat{Q}^{-1}\hat{A}_{12}^{\top}\right)^{-1}\left(\lambda_{p}\hat{A}_{p}v''_{p}-\hat{A}_{12}\hat{Q}^{-1}\mathbb{E}^{\tilde{\mathbb{P}}}\left[\pi\right]\right).\label{eq:bl2}
\end{equation}
Multiplying both sides with $\hat{A}_{p}^{\top}$ from the left and
using the second equation of (\ref{eq:KKTpr-1}), we get 
\begin{equation}
B_{p}^{\top}\eta'_{p}=\hat{A}_{p}^{\top}\left(\hat{A}_{12}\hat{Q}^{-1}\hat{A}_{12}^{\top}\right)^{-1}\left(-\lambda_{p}\hat{A}_{p}v''_{p}+\hat{A}_{12}\hat{Q}^{-1}\mathbb{E}^{\tilde{\mathbb{P}}}\left[\pi\right]\right).\label{eq:nl}
\end{equation}
It is not possible to express $v''_{p}$ from (\ref{eq:nl}) because
$\text{rank}\left(\hat{A}_{p}^{\top}\left(\hat{A}_{12}\hat{Q}^{-1}\hat{A}_{12}^{\top}\right)^{-1}\hat{A}_{p}\right)\leq\left|J\right|\left(\left|L\right|+1\right)+1<\dim W_{p}$
and thus $\hat{A}_{p}^{\top}\left(\hat{A}_{12}\hat{Q}^{-1}\hat{A}_{12}^{\top}\right)^{-1}\hat{A}_{p}$
is not invertible. We can write (\ref{eq:nl}) and the last equation
of (\ref{eq:KKTpr-1}) as the following system of linear equations
\begin{equation}
\left[\begin{array}{cc}
\lambda_{p}\hat{A}_{p}^{\top}\left(\hat{A}_{12}\hat{Q}^{-1}\hat{A}_{12}^{\top}\right)^{-1}\hat{A}_{p} & B_{p}^{\top}\\
\\
B_{p} & 0
\end{array}\right]\left[\begin{array}{c}
v''_{p}\\
\\
\eta'_{p}
\end{array}\right]=\left[\begin{array}{c}
\hat{A}_{p}^{\top}\left(\hat{A}_{12}\hat{Q}^{-1}\hat{A}_{12}^{\top}\right)^{-1}\hat{A}_{12}\hat{Q}^{-1}\mathbb{E}^{\tilde{\mathbb{P}}}\left[\pi\right]\\
\\
b_{p}
\end{array}\right].\label{eq:lin}
\end{equation}

Before we attempt to solve (\ref{eq:lin}), let us try to evaluate
$\frac{\partial v_{p}'}{\partial\mathbb{E}^{\tilde{\mathbb{P}}}\left[\pi\right]}$
with the currently known facts. We get
\[
\begin{array}{rcl}
\frac{\partial v_{p}'}{\partial\mathbb{E}^{\tilde{\mathbb{P}}}\left[\pi\right]} & = & -\frac{1}{\lambda_{p}}\hat{Q}^{-1}-\frac{1}{\lambda_{p}}\hat{Q}^{-1}\hat{A}_{12}^{\top}\frac{\partial\mu'_{p}}{\partial\mathbb{E}^{\tilde{\mathbb{P}}}\left[\pi\right]}\\
\\
 & = & -\frac{1}{\lambda_{p}}\hat{Q}^{-1}+\frac{1}{\lambda_{p}}\hat{Q}^{-1}\hat{A}_{12}^{\top}\left(\hat{A}_{12}\hat{Q}^{-1}\hat{A}_{12}^{\top}\right)^{-1}\left(\hat{A}_{12}\hat{Q}^{-1}-\lambda_{p}\hat{A}_{p}\frac{\partial v''_{p}}{\partial\mathbb{E}^{\tilde{\mathbb{P}}}\left[\pi\right]}\right),
\end{array}
\]
where (\ref{eq:bl1}) and (\ref{eq:bl2}) were used to obtain the
first and the second equality, respectively. We are interested in
the rank of $\hat{A}_{1}\mathcal{D}\tilde{\mathcal{Z}}_{p}\hat{A}_{1}^{\top}$.
Define $\tilde{A}_{1}=\left[\hat{A}_{1}0\right]$, where $0$ is a
matrix of zeros of dimension $\left|J\right|\times\left(N\left|L\right|+1\right)$.
Then,

\begin{equation}
\begin{array}{rcl}
\hat{A}_{1}\mathcal{D}\tilde{\mathcal{Z}}_{p}\hat{A}_{1}^{\top} & = & \tilde{A}_{1}\frac{\partial v_{p}'}{\partial\mathbb{E}^{\tilde{\mathbb{P}}}\left[\pi\right]}\tilde{A}_{1}^{\top}\\
\\
 & = & -\frac{1}{\lambda_{p}}\tilde{A}_{1}\hat{Q}^{-1}\tilde{A}_{1}^{\top}+\frac{1}{\lambda_{p}}\tilde{A}_{1}\hat{Q}^{-1}\hat{A}_{12}^{\top}\left(\hat{A}_{12}\hat{Q}^{-1}\hat{A}_{12}^{\top}\right)^{-1}\hat{A}_{12}\hat{Q}^{-1}\tilde{A}_{1}^{\top}\\
\\
 &  & -\tilde{A}_{1}\hat{Q}^{-1}\hat{A}_{12}^{\top}\left(\hat{A}_{12}\hat{Q}^{-1}\hat{A}_{12}^{\top}\right)^{-1}\hat{A}_{p}\frac{\partial v''_{p}}{\partial\mathbb{E}^{\tilde{\mathbb{P}}}\left[\pi\right]}\tilde{A}_{1}^{\top}.
\end{array}\label{eq:mmm1-1}
\end{equation}
Using the internal structure of $\hat{Q}$ and $\hat{A}_{12}$ as
defined in (\ref{eq:insStr}), we evaluate 
\begin{equation}
\hat{A}_{12}\hat{Q}^{-1}\hat{A}_{12}^{\top}=\left[\begin{array}{cc}
\hat{A}_{1}\hat{Q}_{1}\hat{A}_{1}^{\top} & \hat{A}_{1}\hat{Q}_{2}\hat{A}_{2}^{\top}\\
\hat{A}_{2}\hat{Q}_{2}^{\top}\hat{A}_{1}^{\top} & \hat{A}_{2}\hat{Q}_{3}\hat{A}_{2}^{\top}
\end{array}\right]\label{eq:mn1}
\end{equation}
and similarly
\begin{equation}
\tilde{A}_{1}\hat{Q}^{-1}\hat{A}_{12}^{\top}=\left[\begin{array}{cc}
\hat{A}_{1}\hat{Q}_{1}\hat{A}_{1}^{\top} & \hat{A}_{1}\hat{Q}_{2}\hat{A}_{2}^{\top}\end{array}\right].\label{eq:mn2}
\end{equation}
Using Lemma \ref{prop:Inv}, (\ref{eq:mn1}), and (\ref{eq:mn2}),
we can rewrite (\ref{eq:mmm1-1}) as 
\begin{equation}
\hat{A}_{1}\mathcal{D}\tilde{\mathcal{Z}}_{p}\hat{A}_{1}^{\top}=-\tilde{A}_{1}\hat{Q}^{-1}\hat{A}_{12}^{\top}\left(\hat{A}_{12}\hat{Q}^{-1}\hat{A}_{12}^{\top}\right)^{-1}\hat{A}_{p}\frac{\partial v''_{p}}{\partial\mathbb{E}^{\tilde{\mathbb{P}}}\left[\pi\right]}\tilde{A}_{1}^{\top}.\label{eq:nb}
\end{equation}
In order to evaluate (\ref{eq:nb}), we return back to the system
of linear equations (\ref{eq:lin}). Using the generalized Bott-Duffin
constrained inverse (see \cite{yonglin1990thegeneralized}), we can
express $\frac{\partial v''_{p}}{\partial\mathbb{E}^{\tilde{\mathbb{P}}}\left[\pi\right]}$
from (\ref{eq:lin}) as 
\begin{equation}
\frac{\partial v''_{p}}{\partial\mathbb{E}^{\tilde{\mathbb{P}}}\left[\pi\right]}=\left(\check{A}_{p}\right)_{S}^{\left(+\right)}\hat{A}_{p}^{\top}\left(\hat{A}_{12}\hat{Q}^{-1}\hat{A}_{12}^{\top}\right)^{-1}\hat{A}_{12}\hat{Q}^{-1}+P_{N(\check{A}_{p})\cap S}\frac{\partial z_{p}}{\partial\mathbb{E}^{\tilde{\mathbb{P}}}\left[\pi\right]},\label{eq:mmm}
\end{equation}
where $\left(\check{A}_{p}\right)_{S}^{\left(+\right)}=P_{S}\left(\check{A}_{p}P_{S}+P_{S^{\perp}}\right)^{\left(+\right)}$,
$\check{A}_{p}:=\lambda_{p}\hat{A}_{p}^{\top}\left(\hat{A}_{12}Q_{P}^{-1}\hat{A}_{12}^{\top}\right)^{-1}\hat{A}_{p}$,
$S$ is a null space of $B_{p}$, i.e. $S=N\left(B_{p}\right)$, $P_{S}$
a projection on $S$, and $z_{p}\in\mathbb{R}^{\dim W_{p}}$ an arbitrary
vector. The Moore\textendash{}Penrose pseudoinverse is denoted by
$^{(+)}$. For every $x\in N(\check{A}_{p})$,
\begin{equation}
\begin{array}{rcl}
0 & = & \check{A}_{p}x\\
\\
 & = & \hat{A}_{p}^{\top}\left(\hat{A}_{12}\hat{Q}^{-1}A_{12}^{\top}\right)^{-1}\hat{A}_{p}x\\
\\
 & = & x^{\top}\hat{A}_{p}^{\top}\left(\hat{A}_{12}\hat{Q}^{-1}\hat{A}_{12}^{\top}\right)^{-1}\hat{A}_{p}x\\
\\
 & = & \tilde{Q}\hat{A}_{p}x\\
\\
 & = & \hat{A}_{p}x,
\end{array}
\end{equation}
where $\tilde{Q}\in\mathbb{R}^{\left|J\right|\left(\left|L\right|+1\right)+1\times\left|J\right|\left(\left|L\right|+1\right)+1}$
is defined by the Cholesky decomposition such that
\begin{equation}
\tilde{Q}\tilde{Q}^{\top}=\left(\hat{A}_{12}\hat{Q}^{-1}\hat{A}_{12}^{\top}\right)^{-1}.\label{eq:chol}
\end{equation}
 Since this holds for every $x\in N(\check{A}_{p})$, it must hold
also for every $x\in N(\check{A}_{p})\cap S$ and therefore
\begin{equation}
-\tilde{A}_{1}\hat{Q}^{-1}\hat{A}_{12}^{\top}\left(\hat{A}_{12}\hat{Q}^{-1}\hat{A}_{12}^{\top}\right)^{-1}\hat{A}_{p}P_{N(\check{A}_{p})\cap S}\frac{\partial z_{p}}{\partial\mathbb{E}^{\tilde{\mathbb{P}}}\left[\pi\right]}=0.\label{eq:1}
\end{equation}
Properties about the Bott-Duffin inverse (see \cite{yonglin1990thegeneralized})
show that if $\check{A}_{p}$ is symmetric, then also $\left(\check{A}_{p}\right)_{S}^{\left(+\right)}$
is symmetric. Moreover, if $\left(v''_{p}\right)^{\top}\check{A}_{p}v''_{p}\geq0$
for all $v''_{p}$ such that $B_{p}v''_{p}=b_{p}$, then also $\left(v''_{p}\right)^{\top}\left(\check{A}_{p}\right)_{S}^{\left(+\right)}v''_{p}\geq0$
for all $v''_{p}$. Thus, we can see from (\ref{eq:nb}), (\ref{eq:mmm}),
and (\ref{eq:1}) that $x^{\top}\mathcal{D}\tilde{\mathcal{Z}}_{p}x\leq0$,
$p\in P$ for all $x\in\mathbb{R}^{N}$, which proves the first part
of this theorem.

Using (\ref{eq:1}), (\ref{eq:mn1}) and (\ref{eq:mn2}), equation
(\ref{eq:nb}) reads
\begin{equation}
\hat{A}_{1}\mathcal{D}\tilde{\mathcal{Z}}_{p}\hat{A}_{1}^{\top}=\hat{A}_{3,p}\left(\check{A}_{p}\right)_{S}^{\left(+\right)}\hat{A}_{3,p}^{\top}.
\end{equation}
From \cite{yonglin1990thegeneralized} we know that $R\left(\left(\check{A}_{p}\right)_{S}^{\left(+\right)}\right)=R\left(P_{S}\check{A}_{p}\right)$.
Since $\left(\check{A}_{p}\right)_{S}^{\left(+\right)}$ is symmetric
and positive semidefinite there exists $U\in\mathbb{R}^{\dim W_{p}\times\dim W_{p}}$
such that $\left(\check{A}_{p}\right)_{S}^{\left(+\right)}=UU^{\top}$.
Then,
\begin{equation}
\begin{array}{rcl}
R\left(P_{S}\check{A}_{p}\right) & = & R\left(\left(\check{A}_{p}\right)_{S}^{\left(+\right)}\right)\\
\\
 & = & R\left(UU^{\top}\right)\\
\\
 & = & R(U)
\end{array}\label{eq:2}
\end{equation}
and 
\begin{equation}
\begin{array}{rcl}
\text{rank}\left(\hat{A}_{3,p}\left(\check{A}_{p}\right)_{S}^{\left(+\right)}\hat{A}_{3,p}^{\top}\right) & = & \text{rank}\left(\hat{A}_{3,p}U\right)\\
\\
 & = & \dim R\left(U\right)-\dim N\left(\hat{A}_{3,p}\right)\cap R\left(U\right).
\end{array}\label{eq:dim}
\end{equation}

The next step is to show that $R\left(P_{S}\check{A}_{p}\right)=R\left(P_{S}\hat{A}_{p}^{\top}\right)$.
We know that for any matrices $\tilde{A}$ and $\tilde{B}$ of the
same size $R\left(\tilde{A}\right)=R\left(\tilde{B}\right)$ if and
only if $N\left(\tilde{A}^{\top}\right)=N\left(\tilde{B}^{\top}\right)$.
Since $\left(\hat{A}_{12}\hat{Q}^{-1}\hat{A}_{12}^{\top}\right)^{-1}$
is positive definite and symmetric, there exists an invertible matrix
$\tilde{Q}$ as defined in (\ref{eq:chol}). Then for every $x\in N\left(\left(P_{S}\check{A}_{p}\right)^{\top}\right)$
\begin{equation}
\begin{array}{rcl}
0 & = & \hat{A}_{p}^{\top}\tilde{Q}\tilde{Q}^{\top}\hat{A}_{p}P_{S}x\\
\\
 & = & x^{\top}P_{S}\hat{A}_{p}^{\top}\tilde{Q}\tilde{Q}^{\top}\hat{A}_{p}P_{S}x\\
\\
 & = & \tilde{Q}^{\top}\hat{A}_{p}P_{S}x\\
\\
 & = & \hat{A}_{p}P_{S}x,
\end{array}\label{eq:3}
\end{equation}
and thus $R\left(P_{S}\check{A}_{p}\right)=R\left(P_{S}\hat{A}_{p}^{\top}\right)$.
It follows from (\ref{eq:2}), (\ref{eq:dim}) and (\ref{eq:3}) that
\begin{equation}
\begin{array}{rcl}
\text{rank}\left(\hat{A}_{3,p}\left(\check{A}_{p}\right)_{S}^{\left(+\right)}\hat{A}_{3,p}^{\top}\right) & = & \text{rank}\left(\hat{A}_{3,p}P_{S}\left[\begin{array}{c|c}
\hat{A}_{3,p}^{\top} & \hat{A}_{4,p}^{\top}\end{array}\right]\right)\\
\\
 & = & \text{rank}\left(\hat{A}_{3,p}P_{S}\right),
\end{array}
\end{equation}
where the last equality holds since $P_{S}$ is symmetric. 

A projection matrix $P_{S}$ on the subspace $S=N\left(B_{p}\right)$,
can be written as
\begin{equation}
P_{S}=I-B_{p}^{\top}\left(B_{p}B_{p}^{\top}\right)^{-1}B_{p}.
\end{equation}
By taking into account the structure of $B_{p}$ and $\hat{A}_{3,p}$
is easy to see that $\hat{A}_{3,p}P_{S}$ has full rank if and only
if for each delivery period $j\in J$ there exists at least one power
plant that has a strictly feasible optimal production (i.e. one of
the power plants does not appear in the set of active constraints
$B_{p}$). \end{proof}

It is possible to establish the following unique closed form relation
among the prices $\mathbb{E}^{\tilde{\mathbb{P}}}\left[\Pi\left(t_{i},T_{j}\right)\right]$,
$i\in\left\{ 1,2\right\} $ for any fixed $j\in J$. Since the claim
holds for any fixed $j\in J$ we simplify the notation and avoid writing
$T_{j}$.

\begin{proposition}\label{prop:twoStageMeanVar}If $\Pi\left(t_{1}\right)$
and $\mathbb{E}_{t_{1}}^{\mathbb{P}}\left[\Pi\left(t_{2}\right)\right]$
denote the two-stage equilibrium prices at $t_{1}$, then the following
equality 
\begin{equation}
\begin{array}{rcl}
\Pi\left(t_{1}\right) & = & \mathbb{E}_{t_{1}}^{\mathbb{\tilde{P}}}\left[\Pi\left(t_{2}\right)\right]-\lambda^{-1}\text{Cov}_{t_{1}}\left[\Pi\left(t_{2}\right),\sum_{p\in P}C_{p}\left[k_{p},F_{p}\left(t\right),G\left(t\right)\right]-\left(\sum_{c\in C}s_{c}p_{c}\right)D\right]\end{array}
\end{equation}
where $\lambda^{-1}=\sum_{p=1}^{P}\frac{1}{\lambda_{p}}+\sum_{c=1}^{C}\frac{1}{\lambda_{c}}$
and $k_{p}=V_{p}\left(t_{1}\right)+V_{p}\left(t_{2}\right)$, must
be satisfied. $C_{p}\left(V,F,G\right)$ denotes the costs of producing
$V$ units of electricity with the fuel trading strategy $F$ at fuel
prices $G$. \end{proposition}

\begin{proof}Profit of a producer $p\in P$ can in this setting be
written as
\begin{equation}
\begin{array}{rcl}
P_{p}\left(V_{p}\left(t\right),F_{p}\left(t\right),\Pi\left(t\right)\right) & = & -\Pi\left(t_{1}\right)V_{p}\left(t_{1}\right)-\Pi\left(t_{2}\right)V_{p}\left(t_{2}\right)\\
\\
 &  & -C_{p}\left[V_{p}\left(t_{1}\right)+V_{p}\left(t_{2}\right),F_{p}\left(t\right),G\left(t\right)\right]
\end{array}\label{eq:prod}
\end{equation}
and profit of a consumer $c\in C$ as 
\[
\begin{array}{rcl}
P_{c}\left(V_{c}\left(t\right),\Pi\left(t\right)\right) & = & s_{c}p_{c}D\\
\\
 &  & -\Pi\left(t_{1}\right)V_{c}\left(t_{1}\right)-\Pi\left(t_{2}\right)V_{c}\left(t_{2}\right).
\end{array}
\]

Using Lemma \ref{prop:bounded_volumes}, we can see that for any producer
$p\in P$ there is no constraints on the small changes of $V_{p}\left(t_{1}\right)$
and $V_{p}\left(t_{2}\right)$, but there is a capacity constraint
on changes in the sum $V_{p}\left(t_{1}\right)+V_{p}\left(t_{2}\right)$.
Thus by setting $V_{p}\left(t_{1}\right)+V_{p}\left(t_{2}\right)=k_{p}$
for some fixed $k_{p}\in\mathbb{R}$ for all $p\in P$, constraints
can not be violated by small changes in volumes $V_{p}\left(t_{1}\right)$
or $V_{p}\left(t_{2}\right)$. By inserting (\ref{eq:prod}) into
a mean-variance objective function and taking a derivative, we obtain
the following expression 
\[
\begin{array}{rcl}
\frac{\partial\Psi_{p}\left(P_{p}\left(V_{p}\left(t\right),F_{p}\left(t\right),\Pi\left(t\right)\right)\right)}{\partial V_{p}\left(t_{2}\right)} & = & -\Pi\left(t_{1}\right)+\mathbb{E}_{t_{1}}^{\tilde{\mathbb{P}}}\left[\Pi\left(t_{2}\right)\right]\\
\\
 &  & +\lambda_{p}\left[\text{Var}_{t_{1}}\Pi\left(t_{2}\right)V_{p}\left(t_{2}\right)+\text{Cov}_{t_{1}}\left(\Pi\left(t_{2}\right),C_{p}\left[k_{p},F_{p}\left(t\right),G\left(t\right)\right]\right)\right].
\end{array}
\]
By setting $\frac{\partial\Psi_{p}\left(P_{p}\left(V_{p}\left(t\right),B_{p}\left(t\right),\Pi\left(t\right)\right)\right)}{\partial V_{p}\left(t_{2}\right)}=0$
and expressing $V_{p}\left(t_{2}\right)$, we obtain 
\begin{equation}
V_{p}\left(t_{2}\right)=\frac{\Pi\left(t_{1}\right)-\mathbb{E}_{t_{1}}^{\tilde{\mathbb{P}}}\left[\Pi\left(t_{2}\right)\right]}{\lambda_{p}\text{Var}_{t_{1}}\Pi\left(t_{2}\right)}+\frac{\text{Cov}_{t_{1}}\left[\Pi\left(t_{2}\right),C_{p}\left(k_{p},F_{p}\left(t\right),G\left(t\right)\right)\right]}{\text{Var}_{t_{1}}\Pi\left(t_{2}\right)}.\label{eq:vp}
\end{equation}
Similarly, for each consumer $c\in C$
\begin{equation}
V_{c}\left(t_{2}\right)=\frac{\Pi\left(t_{1}\right)-\mathbb{E}_{t_{1}}^{\tilde{\mathbb{P}}}\left[\Pi\left(t_{2}\right)\right]}{\lambda_{c}\text{Var}_{t_{1}}\Pi\left(t_{2}\right)}+\frac{\text{Cov}_{t_{1}}\left(D,\Pi\left(t_{2}\right)\right)}{\text{Var}_{t_{1}}\Pi\left(t_{2}\right)}.\label{eq:cp}
\end{equation}
The market condition requires
\begin{equation}
0=\sum_{c\in C}V_{c}\left(t_{2}\right)+\sum_{p\in P}V_{p}\left(t_{2}\right).\label{eq:mn}
\end{equation}
Inserting expressions (\ref{eq:vp}) and (\ref{eq:cp}) into (\ref{eq:mn}),
we obtain the result
\[
\begin{array}{rcl}
\Pi\left(t_{1}\right) & = & \mathbb{E}_{t_{1}}^{\tilde{\mathbb{P}}}\left[\Pi\left(t_{2}\right)\right]-\lambda^{-1}\text{Cov}_{t_{1}}\left[\Pi\left(t_{2}\right),\sum_{p\in P}C_{p}\left[k_{p},F_{p}\left(t\right),G\left(t\right)\right]-\left(\sum_{c\in C}s_{c}p_{c}\right)D\right]\end{array}
\]
where $\lambda^{-1}=\sum_{p=1}^{P}\frac{1}{\lambda_{p}}+\sum_{c=1}^{C}\frac{1}{\lambda_{c}}$.
\end{proof}

The analytical expression obtained in Proposition \ref{prop:twoStageMeanVar}
matches expression in \cite{bessembinder2002equilibrium}, \cite{buhler2009valuation},
and \cite{buhler2009riskpremia}, but extends results to more than
one producer and consumer. In contrast to the previous work, it also
allows capacity constraints.

\subsection{Mean Maximization}

In this subsection we would like to examine a mean maximization problem
instead of the mean-variance maximization, in the context of assumptions
and theorems presented in the previous subsection.

A mean maximization optimization problem for each producer $p\in P$
is defined as
\[
\begin{array}{rl}
\underset{v_{p}}{\max} & -\mathbb{E}^{\tilde{\mathbb{P}}}\left[\pi_{p}\right]^{\top}v_{p}\\
\\
\text{s.t.} & A_{p}v_{p}=a_{p}\\
\\
 & B_{p}v_{p}\leq b_{p}
\end{array}
\]
and a mean maximization optimization problem for each consumer $c\in C$
as
\[
\begin{array}{rl}
\underset{V_{c}}{\max} & -\mathbb{E}^{\tilde{\mathbb{P}}}\left[\Pi\right]^{\top}V_{c}\\
\\
\text{s.t.} & A_{c}V_{c}=a_{c}\\
\\
 & B_{c}V_{c}\leq b_{c}.
\end{array}
\]

Under Assumption \ref{ass: as1}, we use Corollary \ref{cor:exist}
to conclude that there exists a solution to the mean maximization
problem.

However, we are not able to claim the uniqueness of solution. Since
$\mathcal{D}_{v_{p}}\Psi_{p}\left(v_{p},\mathbb{E}^{\tilde{\mathbb{P}}}\left[\Pi\right]\right)=0$
for all $p\in P$ and $\mathcal{D}_{V_{c}}\Psi_{c}\left(V_{c},\mathbb{E}^{\tilde{\mathbb{P}}}\left[\Pi\right]^{\top}\right)=0$
for all $c\in C$, expected utility functions of producers and consumers
do not satisfy strict concavity. To get some insight into this problem
let us examine the most simple case where $\left|P\right|=1$, $\left|C\right|=1$,
$T'=1$, and $I_{1}=\left\{ 1\right\} $. We simplify the notation
by setting $V:=V_{p_{1}}$. The sensitivity analysis of the linear
programming shows that the optimal value $V^{*}\left(\mathbb{E}^{\tilde{\mathbb{P}}}\left[\Pi\right]\right)$
for a given price $\mathbb{E}^{\tilde{\mathbb{P}}}\left[\Pi\right]$,
is an increasing piecewise constant function. Since we are looking
for a price $\mathbb{E}^{\tilde{\mathbb{P}}}\left[\Pi\right]^{*}$
such that $V^{*}\left(\mathbb{E}^{\tilde{\mathbb{P}}}\left[\Pi\right]^{*}\right)=D\left(T_{1}\right)$,
one can distinguish two cases (see Figure \ref{fig:MeanEq}). In the
first case (graph on the left) the price equilibrium is not unique,
but the volume for a given price is. In the second case (graph on
the right) the price equilibrium is unique, but the volume for a given
price is not. In order to ensure uniqueness of both, Lemma \ref{prop:uniq_vol-1}
requires uniqueness of volumes for a given price (i.e. forbids that
the graph is vertical), and Theorem \ref{thm:uniq} (by setting $\mathcal{D}\tilde{\mathcal{Z}}\neq0$)
requires uniqueness of prices for a given volume (i.e. forbids that
the graph is horizontal). As we have seen in the previous subsection,
the first condition depends only on the expected utility function,
while the second condition depends on both, the expected utility function
and constraints.

\begin{figure}
\begin{centering}
{\small \begin{tikzpicture}[scale=1.00]
\draw[->] (0,0) -- (4,0) node[anchor=north] {$\mathbb{E}^{\mathbb{P}}\left[\Pi\right]$};
\draw[thick,dashed] (0.5,0) -- (0.5,1.0)  -- (1.5, 1.0) -- (1.5,1.5) -- (2.5, 1.5) -- (2.5, 2.5) -- (3.5, 2.5);
\draw (3.3,2.7) node {$Prod$};

\draw[thick,dotted] (0, 1.5) -- (3.5,1.5);
\draw (3.3,1.65) node {$Con$};

\draw[red,very thick] (1.5, 1.5) -- (2.5,1.5);

\draw[->] (0,0) -- (0,3) node[anchor=east] {$V$};

\begin{scope}[xshift=6.5cm]
\draw[->] (0,0) -- (4,0) node[anchor=north] {$\mathbb{E}^{\mathbb{P}}\left[\Pi\right]$};
\draw[thick,dashed] (0.5,0) -- (0.5,1.0)  -- (1.5, 1.0) -- (1.5,1.5) -- (2.5, 1.5) -- (2.5, 2.5) -- (3.5, 2.5);
\draw (3.3,2.7) node {$Prod$};

\draw[thick,dotted] (0, 1.8) -- (3.5,1.8);
\draw (3.3,1.95) node {$Con$};

\draw[red,very thick,fill] (2.5, 1.8) circle (0.03cm);

\draw[->] (0,0) -- (0,3) node[anchor=east] {$V$};
\end{scope}

\end{tikzpicture}}
\par\end{centering}{\small \par}

\begin{centering}

\par\end{centering}

\centering{}\caption{\label{fig:MeanEq}Equilibrium types for the mean maximization problems.}
\end{figure}
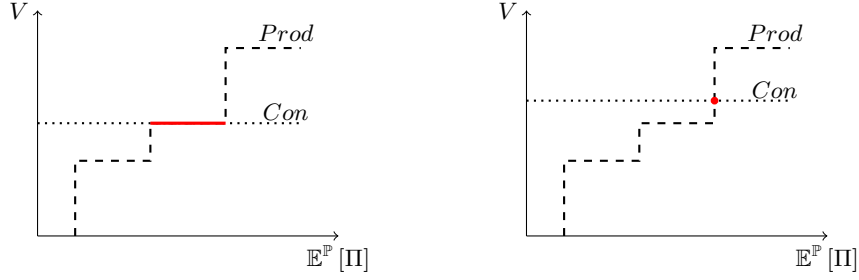

Even though the uniqueness of solution can not be claimed in this
setting, it is still possible to establish the following unique relation
among the prices $\Pi\left(t_{i},T_{j}\right)$, $i\in I_{j}$ for
any fixed $j\in J$. 

\begin{proposition}\label{prop:mean_bounded_vol}Electricity prices
$\mathbb{E}^{\tilde{\mathbb{P}}}\left[\Pi\left(t_{i},T_{j}\right)\right]$
that constitute a CE are equal for all $i\in I_{j}$ and any fixed
$j\in J$. \end{proposition}

\begin{proof}Let there exist $j\in J$, $i'\in I_{j}$, $i''\in I_{j}$,
$i'\neq i''$ such that $\mathbb{E}^{\tilde{\mathbb{P}}}\left[\Pi\left(t_{i'},T_{j}\right)\right]<\mathbb{E}^{\tilde{\mathbb{P}}}\left[\Pi\left(t_{i''},T_{j}\right)\right]$.
Then for each producer $p\in P$, $V_{p}\left(t_{i'},T_{j}\right)\geq V_{p}\left(t_{i''},T_{j}\right)$,
because otherwise they could improve their objective functions by
exchanging $V_{p}\left(t_{i'},T_{j}\right)$ and $V_{p}\left(t_{i''},T_{j}\right)$.
The same also holds for all consumers $c\in C$. Note that (\ref{eq:vol})
implies
\[
0=\sum_{p\in P}V_{p}\left(t_{i'},T_{j}\right)+\sum_{c\in C}V_{c}\left(t_{i'},T_{j}\right)
\]
and 
\[
0=\sum_{p\in P}V_{p}\left(t_{i''},T_{j}\right)+\sum_{c\in C}V_{c}\left(t_{i''},T_{j}\right).
\]
Therefore, $V_{p}\left(t_{i'},T_{j}\right)=V_{p}\left(t_{i''},T_{j}\right)$
for all $p\in P$ and $V_{c}\left(t_{i'},T_{j}\right)=V_{c}\left(t_{i''},T_{j}\right)$
for all $c\in C$.

If $\left|V_{p}\left(t_{i'},T_{j}\right)\right|<V_{trade}$ then this
can not be a CE, because producer $p$ could improve its objective
function by decreasing $V_{p}\left(t_{i'},T_{j}\right)$ for some
$\epsilon>0$ and increasing $V_{p}\left(t_{i''},T_{j}\right)$ for
the same $\epsilon>0$ without changing any other decision variables
or violating any constants. The same reasoning also holds for each
consumer $c\in C$. Thus, $\left|V_{p}\left(t_{i'},T_{j}\right)\right|=V_{trade}$
for all $p\in P$ and $\left|V_{c}\left(t_{i'},T_{j}\right)\right|=V_{trade}$
for all $c\in C$. Such solution clearly does not satisfy (\ref{eq:cb2}).
Thus, $\mathbb{E}^{\tilde{\mathbb{P}}}\left[\Pi\left(t_{i'},T_{j}\right)\right]\geq\mathbb{E}^{\tilde{\mathbb{P}}}\left[\Pi\left(t_{i''},T_{j}\right)\right]$.
One can than apply a similar reasoning to $\mathbb{E}^{\tilde{\mathbb{P}}}\left[\Pi\left(t_{i'},T_{j}\right)\right]>\mathbb{E}^{\tilde{\mathbb{P}}}\left[\Pi\left(t_{i''},T_{j}\right)\right]$
and conclude $\mathbb{E}^{\tilde{\mathbb{P}}}\left[\Pi\left(t_{i'},T_{j}\right)\right]=\mathbb{E}^{\tilde{\mathbb{P}}}\left[\Pi\left(t_{i''},T_{j}\right)\right]$.
\end{proof}

\section{Conclusions\label{sec:Conclusions}}

In this paper we propose a new model for modeling the electricity
price and its relation to other fuels and emission certificates. The
model belongs to a class of game theoretic equilibrium models. In
this paper we rigorously show that there exists a solution to the
proposed model and develop the conditions under which the solution
is also unique. In the last part of the paper we show why the uniqueness
of solution can not be claimed for a mean maximization problem.

\bibliographystyle{Bibtex/siam}
\bibliography{Bibtex/mat_fin_bib,Bibtex/transfer}

\begin{thebibliography}{10}

\bibitem{alexander2006hedging}
{\sc Carol Alexander and Leonardo~M. Nogueira}, {\em Hedging options with
  scale-invariant models}, {ICMA} Centre Discussion Papers in Finance
  icma-dp2006-03, Henley Business School, Reading University, June 2006.

\bibitem{barlow2002adiffusion}
{\sc M.~T. Barlow}, {\em A diffusion model for electricity prices},
  Mathematical Finance, 12 (2002), pp.~287--298.

\bibitem{benth2009theinformation}
{\sc Fred~Espen Benth and Thilo Meyer-Brandis}, {\em The information premium
  for non-storable commodities}, Journal of Energy Markets,  (2009).

\bibitem{berkelaar1996sensitivity}
{\sc A.B. Berkelaar, B.~Jansen, K.~Roos, and T.~Terlaky}, {\em Sensitivity
  analysis in (degenerate) quadratic programming}, tech. report, Jan. 1996.

\bibitem{bessembinder2002equilibrium}
{\sc Hendrik Bessembinder and Michael~L. Lemmon}, {\em Equilibrium pricing and
  optimal hedging in electricity forward markets}, Journal of Finance, 57
  (2002), pp.~1347--1382.

\bibitem{boot1963onsensitivity}
{\sc J.~C.~G. Boot}, {\em On sensitivity analysis in convex quadratic
  programming problems}, Operations Research, 11 (1963), pp.~pp. 771--786.

\bibitem{buhler2009riskpremia}
{\sc Wolfgang B{\"u}hler}, {\em Risk premia of electricity futures: A dynamic
  equilibrium model}, in Risk Management in Commodity Markets, John Wiley \&
  Sons, Ltd., 2009, pp.~61--80.

\bibitem{buhler2009valuation}
{\sc Wolfgang B{\"u}hler and Jens M{\"u}ller-Merbach}, {\em Valuation of
  electricity futures: Reduced-form vs. dynamic equilibrium models}, Mannheim
  Finance Working Paper No. 2007-07,  (2009).

\bibitem{carmona2013electricity}
{\sc Ren{\'e} Carmona, Michael Coulon, and Daniel Schwarz}, {\em Electricity
  price modeling and asset valuation: a multi-fuel structural approach},
  Mathematics and Financial Economics, 7 (2013), pp.~167--202.

\bibitem{carmona2010marketdesign}
{\sc R.~Carmona, M.~Fehr, J.~Hinz, and A.~Porchet}, {\em Market design for
  emission trading schemes}, {SIAM} Review, 52 (2010), pp.~403--452.

\bibitem{cavallo2005electricity}
{\sc Laura Cavallo and Valeria Termini}, {\em Electricity derivatives and the
  spot market in italy. mitigating market power in the electricity market},
  (2005).

\bibitem{clarke1990optimization}
{\sc Frank~H. Clarke}, {\em Optimization and Nonsmooth Analysis}, Classics in
  Applied Mathematics, Society for Industrial and Applied Mathematics, 1990.

\bibitem{clewlow1999amultifactor}
{\sc Les Clewlow and Chris Strickland}, {\em A multi-factor model for energy
  derivatives}, Research Paper Series~28, Quantitative Finance Research Centre,
  University of Technology, Sydney, Dec. 1999.

\bibitem{clewlow1999valuing}
\leavevmode\vrule height 2pt depth -1.6pt width 23pt, {\em Valuing energy
  options in a one factor model fitted to forward prices}, Research Paper
  Series~10, Quantitative Finance Research Centre, University of Technology,
  Sydney, Apr. 1999.

\bibitem{demaeredaertrycke2012liquidity}
{\sc Gauthier De~Maere~d'Aertrycke and Yves Smeers}, {\em Liquidity Risks on
  Power Exchanges: a Generalized Nash Equilibrium model}, 2012.

\bibitem{debreu1952asocial}
{\sc Gerard Debreu}, {\em A social equilibrium existence theorem}, Proceedings
  of the National Academy of Sciences of the United States of America, 38
  (1952), p.~886{\textendash}893.

\bibitem{hambly2009modelling}
{\sc Ben Hambly, Sam Howison, and Tino Kluge}, {\em Modelling spikes and
  pricing swing options in electricity markets}, Quantitative Finance, 9
  (2009), p.~937{\textendash}949.

\bibitem{howison2009stochastic}
{\sc Sam Howison and Michael~C. Coulon}, {\em Stochastic behaviour of the
  electricity bid stack: From fundamental drivers to power prices}, The Journal
  of Energy Markets, 2 (2009).

\bibitem{jeyakumar1998approximate}
{\sc V.~Jeyakumar and D.~T. Luc}, {\em Approximate jacobian matrices for
  nonsmooth continuous maps and c1-optimization}, {SIAM} J. Control Optim., 36
  (1998), p.~1815{\textendash}1832.

\bibitem{lucia2000electricity}
{\sc Julio~J. Lucia and Eduardo Schwartz}, {\em Electricity prices and power
  derivatives: Evidence from the nordic power exchange},  (2000).

\bibitem{ludkovski2011stochastic}
{\sc M.~Ludkovski}, {\em Stochastic switching games and duopolistic competition
  in emissions markets}, {SIAM} Journal on Financial Mathematics, 2 (2011),
  pp.~488--511.

\bibitem{meyer-brandis2008multifactor}
{\sc Thilo Meyer-Brandis and Peter Tankov}, {\em Multi-factor jump-diffusion
  models of electricity prices}, International Journal of Theoretical and
  Applied Finance ({IJTAF}), 11 (2008), pp.~503--528.

\bibitem{ralph1997sensitivity}
{\sc Daniel Ralph and Stefan Scholtes}, {\em Sensitivity analysis of composite
  piecewise smooth equations}, Math. Program., 76 (1997),
  p.~593{\textendash}612.

\bibitem{robinson2005mathmodel}
{\sc Sara Robinson}, {\em Math model explains high prices in electricity
  markets}, {SIAM} News,  (2005).

\bibitem{yonglin1990thegeneralized}
{\sc Chen Yonglin}, {\em The generalized bott-duffin inverse and its
  applications}, Linear Algebra and its Applications, 134 (1990), pp.~71 -- 91.

\end{thebibliography}

\end{document}